\documentclass[a4paper, 10pt, final]{amsart}

\usepackage{xcolor}

\usepackage{amsthm}

\theoremstyle{plain}

\usepackage[latin1]{inputenc}
\usepackage[T1]{fontenc}
\usepackage[english]{babel}
\usepackage{fullpage}
\usepackage{amsmath}
\usepackage{dsfont}\let\mathbb\mathds
\usepackage{setspace}
\usepackage{amsfonts}
\usepackage{amssymb}
\usepackage{graphicx}
\usepackage{amscd}
\usepackage{mathrsfs}
\usepackage{fancybox}

\usepackage{lmodern}
\usepackage{graphicx}
\usepackage{fancyheadings}
\usepackage{chngcntr}
\counterwithout{figure}{section}
\usepackage{tikz}
\usetikzlibrary{calc}
\usepackage{pgfplots}
\usepgfplotslibrary{groupplots}
\usepgfplotslibrary{units}

\usetikzlibrary{arrows}
\usepackage{tkz-berge} 

\usetikzlibrary{graphs,shapes.geometric,arrows,fit,matrix,positioning}

\usepackage[extra]{tipa}

\usepackage[OT2,T1]{fontenc}

\usepackage{amsmath, amsthm, amsfonts}
\cornersize{2}
\usepackage{graphicx}

\DeclareMathOperator{\Lie}{Lie}
\DeclareMathOperator{\weight}{weight}
\DeclareMathOperator{\id}{id}

\DeclareMathOperator{\pr}{pr}

\DeclareMathOperator{\Frac}{Frac}

\DeclareMathOperator{\reg}{reg}
\DeclareMathOperator{\Ad}{Ad}
\DeclareMathOperator{\depth}{depth}

\DeclareMathOperator{\DR}{DR}
\DeclareMathOperator{\Int}{Int}

\DeclareMathOperator{\comm}{comm}

\DeclareMathOperator{\Spec}{Spec}
\DeclareMathOperator{\un}{un}

\DeclareMathOperator{\inv}{inv}

\DeclareMathOperator{\crys}{crys}

\DeclareMathOperator{\eval}{eval}
\DeclareMathOperator{\har}{har}

\DeclareMathOperator{\iter}{iter}

\DeclareMathOperator{\Li}{Li}
\DeclareMathOperator{\Sign}{Sign}

\DeclareMathOperator{\loc}{loc}

\DeclareMathOperator{\Hom}{Hom}

\DeclareMathOperator{\Map}{Map}

\DeclareMathOperator{\Wd}{Wd}

\DeclareMathOperator{\KZ}{KZ}

\DeclareMathOperator{\dec}{dec}

\newtheorem{Theorem}{Theorem}[section]
\newtheorem{Proposition}[Theorem]{Proposition}

\newtheorem{Lemma}[Theorem]{Lemma}

\newtheorem{Definition}[Theorem]{Definition}

\newtheorem{Proposition-Definition}[Theorem]{Proposition-Definition}

\newtheorem{Lemma-Definition}[Theorem]{Lemma-Definition}

\newtheorem{Example}[Theorem]{Example}

\newtheorem{Notation}[Theorem]{Notation}

\newtheorem{Convention}[Theorem]{Convention}

\newtheorem{Nota Bene}[Theorem]{Nota Bene}

\input cyracc.def
\DeclareFontFamily{U}{russian}{}
\DeclareFontShape{U}{russian}{m}{n}
        { <5><6> wncyr5
        <7><8><9> wncyr7
        <10><10.95><12><14.4><17.28><20.74><24.88> wncyr10 }{}
\DeclareSymbolFont{Russian}{U}{russian}{m}{n}
\DeclareSymbolFontAlphabet{\mathcyr}{Russian}
\makeatletter
\let\@math@cyr\mathcyr
\renewcommand{\mathcyr}[1]{\@math@cyr{\cyracc #1}}
\makeatother
\newcommand{\sh}{\mathcyr{sh}} 

\makeatletter
\newcounter{subsubsubsection}[subsubsection]
\renewcommand\thesubsubsubsection{\thesubsubsection .\@alph\c@subsubsubsection}
\newcommand\subsubsubsection{\@startsection{subsubsubsection}{4}{\z@}%
                                     {-3.25ex\@plus -1ex \@minus -.2ex}%
                                     {1.5ex \@plus .2ex}%
                                     {\normalfont\normalsize\bfseries}}
\newcommand*\l@subsubsubsection{\@dottedtocline{3}{10.0em}{4.1em}}
\newcommand*{\subsubsubsectionmark}[1]{}
\makeatother

\numberwithin{equation}{section}

\usepackage[top=2.7cm, bottom=2.7cm, left=2.7cm, right=2.7cm]{geometry}

\author{David Jarossay}

\address{Universit\'{e} de Gen\`{e}ve, Section de mathématiques, 2-4 rue du Li`{e}vre,
	Case postale 64
	1211 Gen\`{e}ve 4, Suisse}
\email{david.jarossay@unige.ch}
\begin{document}

\title{$p$-adic multiple zeta values at roots of unity 
	\\ $\&$ $p$-adic pro-unipotent harmonic actions
\\ \text{ }
\\ IV : Around $p$-adic continuity and interpolation in $\mathbb{Z}_{p}^{\depth}$
\\ \text{ }
\\ IV-1 : $p$-adic multiple zeta values at roots of unity extended to sequences of integers of any sign}

\maketitle

\noindent

\begin{abstract}
This work is a study of $p$-adic multiple zeta values at roots of unity ($p$MZV$\mu_{N}$'s), the $p$-adic periods of the crystalline pro-unipotent fundamental groupoid of $(\mathbb{P}^{1} - \{0,\mu_{N},\infty\})/ \mathbb{F}_{q}$. The main tool is new objects which we call $p$-adic pro-unipotent harmonic actions. In this part IV we define and study $p$-adic analogues of some elementary complex analytic functions which interpolate multiple zeta values at roots of unity such as the multiple zeta functions. The indices of $p$MZV$\mu_{N}$'s involve sequences of positive integers ; in this IV-1, by considering an operation which we call localization (inverting certain integration operators) in the pro-unipotent fundamental groupoid of $\mathbb{P}^{1} - \{0,\mu_{N},\infty\}$, and by using $p$-adic pro-unipotent harmonic actions, we extend the definition of $p$MZV$\mu_{N}$'s to indices for which these integers can be negative, and we study these generalized  $p$MZV$\mu_{N}$'s.
\end{abstract}

\tableofcontents

\section{Introduction}

\subsection{Complex and $p$-adic multiple zeta values at roots of unity}

This work is a study of $p$-adic multiple zeta values at roots of unity ($p$MZV$\mu_{N}$'s), the $p$-adic periods of the crystalline pro-unipotent fundamental groupoid (abbreviated $\pi_{1}^{\un,\crys}$) of $(\mathbb{P}^{1} - \{0,\mu_{N},\infty\})/ \mathbb{F}_{q}$, with $\mathbb{F}_{q}$ of characteristic $p$ prime to $N$ which contains a primitive $N$-th root of unity. They are $p$-adic analogues of the (complex) multiple zeta values at roots of unity (MZV$\mu_{N}$'s), which are the following numbers. Let $\tilde{\xi}_{N} \in \mathbb{C}$ be a primitive $N$-th root of unity, $n_{1},\ldots,n_{d} \in \mathbb{N}^{\ast}$, and $j_{1},\ldots,j_{d} \in \mathbb{Z}/N\mathbb{Z}$ such that $(n_{d},j_{d}) \not= (1,0)$, and $(z_{i_{1}},\ldots,z_{i_{n}})=(\xi^{j_{1}},\overbrace{0,\ldots,0}^{n_{1}-1},\ldots,\xi^{j_{d}},\overbrace{0,\ldots,0}^{n_{d}-1})$, the associated multiple zeta value at $N$-th roots of unity is :
\begin{equation}
\label{eq:multizetas}
\zeta
\big((n_{i});(\tilde{\xi}_{N}^{j_{i}})\big)_{d} = (-1)^{d}\int_{0}^{1} \frac{d t_{n}}{t_{n} - z_{i_{n}}} \int_{0}^{t_{n-1}} \ldots \int_{0}^{t_{2}} \frac{dt_{1}}{t_{1} - z_{i_{1}}} = \sum_{0<m_{1}<\ldots<m_{d}} \frac{\big( \frac{\tilde{\xi}_{N}^{j_{2}}}{\tilde{\xi}_{N}^{j_{1}}}\big)^{m_{1}} \ldots \big( \frac{1}{\tilde{\xi}_{N}^{j_{d}}}\big)^{m_{d}}}{m_{1}^{n_{1}} \ldots m_{d}^{n_{d}}} \in \mathbb{C}
\end{equation}
where $n=n_{d}+\ldots+n_{1}$ is called the weight of $\big((n_{i});(\tilde{\xi}_{N}^{j_{i}})\big)_{d}$, and $d$ is called its depth.
\newline\indent The $p$MZV$\mu_{N}$'s (Definition \ref{MZV Deligne}) are defined abstractly, without explicit formulas, as $p$-adic analogues of the iterated path integrals in (\ref{eq:multizetas}). They are numbers $\zeta_{p,\alpha}\big((n_{i});(\xi_{N}^{j_{i}})\big)_{d} \in K = \Frac(W(\mathbb{F}_{q})) \subset \overline{\mathbb{Q}_{p}}$, where $\xi_{N} \in K$ is a primitive $N$-th root of unity, $\alpha \in (\mathbb{Z}-\{0\}) \cup \{\pm\infty\}$ represents the number of iterations of the Frobenius of $\pi_{1}^{\un,\crys}((\mathbb{P}^{1} - \{0,\mu_{N},\infty\})/ \mathbb{F}_{q})$, with $\mathbb{F}_{q}$, $d \in \mathbb{N}^{\ast}$ $n_{1},\ldots,n_{d} \in \mathbb{N}^{\ast}$, $j_{1},\ldots,j_{d} \in \mathbb{Z}/N\mathbb{Z}$. 
\newline\indent In the terminologies above, the term "at roots of unity" is usually omitted if $N=1$.

\subsection{Summary of parts I, II, III}

In part I \cite{I-1} \cite{I-2} \cite{I-3}, we have found a $p$-adic analogue of the series formula in (\ref{eq:multizetas}) which has two particular features : it is explicit and it keeps track of the motivic Galois action on $\pi_{1}^{\un,\crys}((\mathbb{P}^{1} - \{0,\mu_{N},\infty\})/ \mathbb{F}_{q})$.
\newline\indent In part II \cite{II-1} \cite{II-2} \cite{II-3}, we have deduced from these formulas a version of the motivic Galois theory of $p$MZV$\mu_{N}$'s formulated in terms of series instead of being formulated as usual in terms of integrals.
\newline\indent In part III \cite{III-1} \cite{III-2}, we have defined and studied a generalization of the notion of $p$MZV$\mu_{N}$'s at roots of unity of order divisible by $p$.
\newline\indent In all the previous parts, our results led us to replace the $p$MZV$\mu_{N}$'s by some variants, equivalent to them in a certain sense, which we call the adjoint $p$-adic multiple zeta values at $N$-th roots of unity (Definition \ref{def adjoint MZV}), abbreviated Ad$p$MZV$\mu_{N}$'s. Indeed, the results of \cite{I-2} \cite{I-3} show that the Ad$p$MZV$\mu_{N}$'s are more directly adapted to explicit computations than the $p$MZV$\mu_{N}$'s. The main objects in the previous parts were some group actions called $p$-adic pro-unipotent harmonic actions found in \cite{I-2} and \cite{I-3}.
\newline\indent The explicit formulas for $p$MZV$\mu_{N}$'s found in part I appear as relations between Ad$p$MZV$\mu_{N}$'s and the following numbers called weighted multiple harmonic sums, with $m \in \mathbb{N}^{\ast}$, $d \in \mathbb{N}^{\ast}$ $n_{1},\ldots,n_{d} \in \mathbb{N}^{\ast}$, $j_{1},\ldots,j_{d+1} \in \mathbb{Z}/N\mathbb{Z}$ \footnote{For complex or $p$-adic MZV$\mu_{N}$'s, $\big((n_{i});(\tilde{\xi}_{N}^{j_{i}})\big)_{d}$ is an abbreviation of $\big( \begin{array}{cc} \xi_{N}^{j_{1}},\ldots,\xi_{N}^{j_{d}} \\ n_{1},\ldots,n_{d} \end{array} \big)$, whereas for multiple harmonic sums, it is an abbreviation of $\big( \begin{array}{cc} \xi_{N}^{j_{1}},\ldots,\xi_{N}^{j_{d+1}} \\ n_{1},\ldots,n_{d} \end{array} \big)$} :
$$ \har_{m} \big((n_{i});(\tilde{\xi}_{N}^{j_{i}})\big)_{d} = \sum_{0<m_{1}<\ldots<m_{d}<m} \frac{\big( \frac{\xi_{N}^{j_{2}}}{\xi_{N}^{j_{1}}}\big)^{m_{1}} \ldots \big( \frac{\xi_{N}^{j_{d+1}}}{\xi_{N}^{j_{d}}}\big)^{m_{d}}\big(\frac{1}{\xi_{N}^{j_{d+1}}}\big)^{m}}{m_{1}^{n_{1}} \ldots m_{d}^{n_{d}}} \in \mathbb{Q}(\xi) $$

\subsection{Motivation for part IV}

The MZV's of depth one, i.e. the numbers (\ref{eq:multizetas}) with $N=1$ and $d=1$, are the values of Riemann's zeta function at positive integers, and special values of other classical functions appearing in analytic number theory. These functions have generalizations which depend on any $d\in \mathbb{N}^{\ast}$ variables, defined some iterated series as in (\ref{eq:multizetas}), as well as generalizations "at roots of unity" taking into account the numerators in the iterated series of (\ref{eq:multizetas}).
\newline Since we constructed in part I, and studied in parts II and III, a $p$-adic analogue of the series expansions of MZV$\mu_{N}$'s, we want to know if $p$-adic analogues of such interpolating functions exist.

\subsection{Motivation for parts IV-1}

We take as starting point of this paper the most straightforward example of interpolation of the MZV$\mu_{N}$'s ; namely, the multiple zeta functions at $N$-th roots of unity (MZF$\mu_{N}$'s) : for any $j_{1},\ldots,j_{d} \in \mathbb{Z}/N\mathbb{Z}$ :
$$ (s_{1},\ldots,s_{d} \big) \in U_{d} \mapsto \sum_{0<m_{1}<\ldots<m_{d}} \frac{\big( \frac{\tilde{\xi}_{N}^{j_{2}}}{\tilde{\xi}_{N}^{j_{1}}}\big)^{m_{1}} \ldots \big( \frac{1}{\tilde{\xi}_{N}^{j_{d}}}\big)^{m_{d}}}{m_{1}^{s_{1}} \ldots m_{d}^{s_{d}}} \in \mathbb{C} $$
\noindent where $U_{d}= \{(s_{1},\ldots,s_{d}) \in \mathbb{C}^{d} \text{ }|\text{ }\text{Re}(s_{d-r+1}+\ldots+s_{d})> r\text{ for all } r=1,\ldots,d\}$. We want to know whether these functions have natural $p$-adic analogues interpolating $p$MZV$\mu_{N}$'s and, if this is the case, we want to study them.
\newline \indent
The MZF$\mu_{N}$'s have a meromorphic continuation to $\mathbb{C}^{d}$, defined in \cite{Essouabri}, \cite{Matsumoto}, \cite{Zhao}, \cite{AET}, \cite{Go} for $N=1$, and in \cite{FKMT1} for any $N$. Their meromorphic continuation has singularities along certain hyperplanes, which are identified the most precisely in \cite{AET} and \cite{FKMT1}.
\newline One can then define values of these functions at tuples of integers of any sign, i.e. $\zeta\big(\begin{array}{c} \tilde{\xi}_{N}^{j_{1}},\ldots,\tilde{\xi}_{N}^{j_{d}} \\ n_{1},\ldots,n_{d} \end{array} \big)$, for any $n_{1},\ldots,n_{d} \in \mathbb{Z}$ ; this requires to remove a singularity. This can be done by considering the limit at tuples of integers along a certain direction \cite{AET} \cite{AT} \cite{Komori} \cite{O} \cite{Sa1} \cite{Sa2}, or by a certain "renormalization" process \cite{GZ}, \cite{MP} \cite{GPZ}, or by a certain "desingularization" process \cite{FKMT1}.
\newline We note that the question of defining numbers $\zeta\big((n_{i}),(\tilde{\xi}_{N}^{j_{i}})\big)$, for any $n_{1},\ldots,n_{d} \in \mathbb{Z}$ does not necessarily involves the meromorphic continuation of MZF$\mu_{N}$'s ; one can consider it only via the formulas of equation (\ref{eq:multizetas}) : we have to find a correct notion of regularizations, either for a generalization of the expression of MZV$\mu_{N}$'s as iterated series :
\begin{equation} \label{eq: complex limit series} \zeta_{\text{reg},\Sigma}\big( (\pm n_{i}),(\tilde{\xi}_{N}^{j_{i}}) \big) = \reg \lim_{m \rightarrow \infty} \sum_{0<m_{1}<\ldots<m_{d}<m} \frac{\big( \frac{\tilde{\xi}_{N}^{j_{2}}}{\tilde{\xi}_{N}^{j_{1}}}\big)^{m_{1}} \ldots \big( \frac{1}{\tilde{\xi}_{N}^{j_{d}}}\big)^{m_{d}}}{m_{1}^{\pm n_{1}}\ldots m_{d}^{\pm n_{d}}} 
\end{equation}
\noindent or for a generalization of the expression of MZV$\mu$s as iterated integrals, which amounts to :
\begin{equation} \label{eq: complex limit integrals}
\zeta_{\text{reg},\smallint}\big((\pm n_{i}),(\tilde{\xi}_{N}^{j_{i}})\big) = \reg\lim_{z \rightarrow 1} \sum_{0<m_{1}<\ldots<m_{d}} \frac{\big( \frac{\tilde{\xi}_{N}^{j_{2}}}{\tilde{\xi}_{N}^{j_{1}}}\big)^{m_{1}} \ldots \big( \frac{z}{\tilde{\xi}_{N}^{j_{d}}}\big)^{m_{d}}}{m_{1}^{\pm n_{1}}\ldots m_{d}^{\pm n_{d}}} 
\end{equation}
\noindent In the end, there are several notions of $\zeta\big((n_{i});(\tilde{\xi}_{N}^{j_{i}})\big)_{d}$ for any $n_{1},\ldots,n_{d} \in \mathbb{Z}$, interrelated in various ways. We would like to know if there are natural $p$-adic analogues of these values.
\newline\indent
The $p$-adic zeta function of Kubota and Leopoldt, which we will denote by $L_{p}$, is defined as a $p$-adic interpolation of the desingularized values of the Riemann zeta function at positive integers, using Kummer's congruences. Coleman has proved in \cite{Coleman} that, for all $n \in \mathbb{N}^{\ast}$ such that $n \geq 2$, we have $\zeta_{p,1}(n) = p^{n} L_{p}(n,\omega^{1-n})$, where $\omega$ is Teichm\"{u}ller's character, and $\zeta_{p,1}$ refers to $p$MZV$\mu_{N}$'s as denoted in \S1.1. This implies that the map $n \in \mathbb{N}^{\ast} \subset \mathbb{Z}_{p} \mapsto p^{-n}\zeta_{p,1}(n) \in \mathbb{Q}_{p}$ is continuous with respect to $n$ on each class of congruence modulo $p-1$, and can be extended to a continuous function on $\mathbb{Z}_{p}$, except for the class of congruence $n \equiv 1 \mod p-1$, where this is true instead for $n \mapsto (n-1)\zeta_{p}(n)$. This property can be retrieved by the following formula, which is known, and is also a particular case of our formulas of part I for $p$MZV$\mu_{N}$'s :
\begin{equation} \label{eq:series expansion depth one} \zeta_{p,1}(n) = \frac{1}{n-1} \sum_{l\geq -1}{-n \choose l} B_{l} \sum_{0<m<p}\frac{p^{n+l-1}}{m^{n+l-1}}
\end{equation}
This gives hope that $p$MZV$\mu_{N}$'s of higher depth might have some $p$-adic continuity properties and might be interpolated by a continuous function. However, the case of depth one is particular : the Frobenius of $\pi_{1}^{\un,\crys}(\mathbb{P}^{1} - \{0,\mu_{N},\infty\})$ can be described as a relation between certain $p$-adic series indexed as $\sum_{0<m}$ and their variants restricted to $\sum_{\substack{0<m \\ p\nmid m}}$ ; in higher depth, the Frobenius is much more complicated.
\newline\indent In higher depth, actually, some generalizations of the Kubota-Leopoldt $L$-function, based on the desingularization of the meromorphic continuation of MZF$\mu_{N}$'s, have been defined \cite{FKMT2}. Some of their values at tuples of positive integers are expressed in terms of $p$-adic iterated integrals on $\mathbb{P}^{1} - \{0,\mu_{cp},\infty\}$, with $c \in \mathbb{N}^{\ast}$ prime to $p$ (\cite{FKMT2}, Theorem 3.41). 
However, studying the role of these functions in the question explained in \S1.3 goes beyond the scope of this paper.
\newline\indent In \cite{FKMT3}, some $p$MZV$\mu_{N}$'s at tuples of integers of any sign are defined in certain particular cases : the indices are $\big((n_{i});(\xi_{N}^{j_{i}})_{d}\big)$ for any $n_{1},\ldots,n_{d} \in \mathbb{Z}$, but  $\xi_{N}^{j_{1}}\not=1,\ldots,\xi_{N}^{j_{d}}\not=1$. 
In terms of our notation of \S1.1, they are generalizations of the values $\zeta_{p,-\infty}(w)$. Their definition relies on the theory of Coleman integration in the sense of \cite{Vologodsky}.

\subsection{Main ideas}

In the formula (\ref{eq:multizetas}), the exponent $n_{i}$ in the iterated series corresponds to the iteration $n_{i}-1$ times of the operator $f \mapsto \int f \frac{dz}{z}$ in the definition of the iterated integral. Replacing this integration operator by its inverse gives similar series (provided it converges) with $n_{i}$ possibly negative. This has been used in several papers, including in \cite{FKMT3}, and we also have used it in \cite{I-2}. We are going to use again this idea here, but more systematically. We will use the term "localization of $\pi_{1}^{\un,\DR}(\mathbb{P}^{1} - \{0,\mu_{N},\infty\})$" to refer to the inversion of the all the operators $f \mapsto \int f \omega$ with $\omega$ a differential form $\frac{dz}{z-x}$, $x \in \{0,\xi^{1},\ldots,\xi^{N}\}$, on $\mathbb{P}^{1} - \{0,\mu_{N},\infty\}$.
\newline\indent 
If we consider an iterated integral as in (\ref{eq:multizetas}) but on a variable path (in the sense of \cite{Chen}), instead of the path $[0,1] \rightarrow [0,1]$, $t \mapsto t$ which is implicit in (\ref{eq:multizetas}), we obtain functions called multiple polylogarithms \cite{Go}, characterized as solutions to a certain differential equation (Proposition-Definition \ref{prop connexion}). If we allow the inversion of integration operators, we will obtain "localized multiple polylogarithms", which will be $\mathbb{Q}(\xi)$-linear combinations of products of iterated integrals by algebraic functions.  This phenomenon already appeared implicitly in \cite{I-2}, \S4-\S5, via the map "loc" which was used to define what we called the $p$-adic pro-unipotent $\Sigma$-harmonic action. Here, this phenomenon will be studied intrinsically, in particular its $p$-adic aspects. By this phenomenon, all the numbers obtained by considering localized iterated integrals at tangential base-points remain in the same algebra of periods : they are certain $\mathbb{Q}(\xi)$-linear combinations of $p$MZV$\mu_{N}$'s.
\newline\indent 
The $p$MZV$\mu_{N}$'s are defined using the notion of Frobenius structure of a $p$-adic differential equation. Thus in order to look for a good meaning of a notion of $p$MZV$\mu_{N}$'s at sequences of integers of any sign, we should show a compatibility between the localization and the Frobenius structure. We will see that imposing this compatibility makes things actually simpler in the $p$-adic case than in the complex case.
\newline\indent Our idea is to replace the Frobenius by what we called the harmonic Frobenius in \cite{I-2}, and to use $p$-adic pro-unipotent harmonic actions. Although it is possible to define localized $p$-adic multiple polylogarithms by Coleman integration, their regularization at $z\rightarrow 1$ is not well-defined in general because they have a pole which is not logarithmic. This is the difficulty observed in \cite{FKMT3}. We will see how to avoid it by replacing the Frobenius by the harmonic Frobenius.
\newline\indent 
In part I, we have obtained formulas involving series, representing the $p$MZV$\mu_{N}$'s. However, both the domains of summation and the summands were functions of the indices $n_{d},\ldots,n_{1}$. Here, if we want to study these sums of series as functions of $n_{d},\ldots,n_{1}$, we will make some changes of variables giving domains of summation independent of $n_{d},\ldots,n_{1}$.

\subsection{Outline}

In \S2, we define the "localization" of  $\pi_{1}^{\un,\DR}(X)$ for $X$ equal to a punctured projective line over a field of characteristic zero and the $\KZ$ connection associated with it, on a neighborhood of $0$ (Definition \ref{def of the localization}). 
This incorporates a notion of localized multiple polylogarithms (Definition \ref{localized Li}).
\newline We define maps which encode different expressions of the "localized iterated integrals" in terms of algebraic functions and the iterated integrals (Proposition-Definition \ref{loc for Li}, Proposition-Definition \ref{loc for har n}).
\newline Then we define the analytic continuation of the localized multiple polylogarithms, in the complex setting (Definition \ref{analytic continuation loc MPL}) and in the $p$-adic setting (Definition \ref{p-adic continuation loc MPL}). The $p$-adic setting is applied to $\pi_{1}^{\un,\crys}(\mathbb{P}^{1} - \{0,\mu_{N},\infty\}/\mathbb{F}_{q})$.
\newline\indent In \S3 we review the notion of $p$MZV$\mu_{N}$'s (Definition \ref{MZV Deligne}, Definition \ref{MZV Coleman},  Definition \ref{MZV Coleman bis}), and the $p$-adic pro-unipotent $\Sigma$-harmonic action $\circ_{\har}^{\Sigma}$ of \cite{I-2}. We define some "localized" variants of $\circ_{\har}^{\Sigma}$ (Proposition-Definition \ref{loc action}) and the localized adjoint $p$MZV$\mu_{N}$'s (Definition \ref{def localized pMZV}).  In this new setting, the localized version of $\circ_{\har}^{\Sigma}$ defined in \cite{I-2} will be now viewed as "localized at the source", and we now have two other "localized" variants of $\circ_{\har}^{\Sigma}$, called, respectively, "localized at the target" and "localized at the source and target".
\newline\indent In \S4, we explain briefly why the properties from \cite{I-3} describing formulas for the iteration of the harmonic Frobenius can be generalized to the setting of \S3 (Proposition-Definition \ref{iter localized}).
\newline\indent In \S5, we bring together the localization built in \S3 and the algebraic relations satisfied by $p$MZV$\mu_{N}$'s, which we studied in part II \cite{II-1} \cite{II-2} \cite{II-3} ; and we show that the localized adjoint $p$MZV$\mu_{N}$'s satisfy a variant of the adjoint double shuffle relations defined in \cite{II-1} (Proposition \ref{localized adjoint quasi shuffle}).
\newline\indent The main theorem is a summary of the main properties of our localized $p$-adic multiple zeta values at roots of unity.
\newline We refer to the notion of adjoint double shuffle relations defined in \cite{II-1}.
\newline 
\newline \textbf{Theorem IV-1}
\newline \emph{i) (Nature of localized Ad$p$MZV$\mu_{N}$'s)
\newline The localized Ad $p$MZV$\mu_{N}$'s are in the $\mathbb{Q}(\xi_{N})$-algebra generated by $p$MZV$\mu_{N}$'s. In particular, they are periods of the crystalline pro-unipotent fundamental groupoid of $(\mathbb{P}^{1} - \{0,\mu_{N},\infty\})/\mathbb{F}_{q}$.
\newline The totally negative Ad $p$MZV$\mu_{N}$'s (in the sense of Definition \ref{def totally negative}) are algebraic numbers, in $\mathbb{Q}(\xi)$; more precisely, they are in an algebra of functions defined explicitly in terms of polynomials of Bernoulli numbers, or, alternatively, in terms of prime multiple harmonic sums at negative indices. The vanishing of the odd Bernoulli numbers imply the vanishing of certain particular totally negative Ad$p$MZV$\mu_{N}$'s.
\newline ii) (Formulas)
\newline The formulas of \cite{I-2} and \cite{I-3} for Ad$p$MZV$\mu_{N}$'s can be extended into explicit formulas for the localized Ad $p$MZV$\mu_{N}$'s, involving extensions of the $p$-adic pro-unipotent harmonic action $\circ_{\har}^{\Sigma}$ and the map of iteration of the harmonic Frobenius $\iter_{\har}^{\Sigma}$.
\newline iii) (Algebraic relations)
\newline The localized Ad $p$MZV$\mu_{N}$'s satisfy an extension of the adjoint double shuffle relations}
\newline 
\newline We consider these properties as a justification of our definition of the localized Ad$p$MZV$\mu_{N}$'s.
\newline We note that this theorem can be extended in an obvious way to the $p$-adic multiple zeta values at roots of unity of order divisible by $p$ defined in \cite{III-1}.
\newline\indent We will prove later that the formulas of \cite{I-2} can be modified to be formulas with domains of summation independent of $(l,(n_{i});(\xi_{N}^{j_{i}}))$.
\newline We will use this later to deduce that the adjoint $p$MZV's $\zeta_{p}^{\Ad}(l;n_{d},\ldots,n_{1})$  (here $N=1$) have some continuity properties with respect to $n_{1}$ and $n_{d}$ viewed as $p$-adic integers.
\newline In the next version of this paper, we will also define a notion of localized $p$MZV$\mu_{N}$'s, without the term adjoint. This will be done by using a generalization of the main equation found in \cite{AETbis}.
\newline 
\newline \textbf{Acknowledgments.} I thank Hidekazu Furusho for discussions.
\newline The idea of inverting integration operators being standard, they may be some references missing and I apologize if this is the case.
\newline This work has been achieved at Universit\'{e} de Strasbourg, supported by Labex IRMIA, and at Universit\'{e} de Genève, supported by NCCR SwissMAP.

\section{Localized complex and $p$-adic multiple polylogarithms}

We formalize a notion of localization on a neighborhood of 0 of the De Rham pro-unipotent fundamental groupoid of and the KZ connection (\S2.1), we define "localization maps" enabling to "compute" it (\S2.2) and we discuss the analytic continuation of the localized multiple polylogarithms (\S2.3) in the complex (\S2.3.1) and $p$-adic (\S2.3.2) settings. The $p$-adic aspects are restricted to the case of $\mathbb{P}^{1} - \{0,\mu_{N},\infty\}$.
 
\subsection{The De Rham pro-unipotent fundamental groupoid of $\mathbb{P}^{1} - \{0=z_{0},z_{1},\ldots,z_{r},\infty\}$, the KZ connection, and its localization on a neighborhood of the origin}

\subsubsection{Review on $\pi_{1}^{\un,\DR}(\mathbb{P}^{1} - \{0=z_{0},z_{1},\ldots,z_{r},\infty\})$ and $\nabla_{\KZ}$}

Let $K$ be a field of characteristic zero, $z_{0},\ldots,z_{r} \in K$ with $z_{0}=0$ and $z_{r}=1$, and 
$X= (\mathbb{P}^{1} - \{0,z_{1},\ldots,z_{r},\infty\})/K$.
\newline We review $\pi_{1}^{\un,\DR}(X)$ and the KZ connection (defined and described in \cite{Deligne}), and we define their "localized" version on a neighborhood of $0$.
\newline\indent The De Rham pro-unipotent fundamental groupoid of $X$, denoted by $\pi_{1}^{\un,\DR}(X)$ is a groupoid in pro-affine schemes over $X$ ; its base-points are the points of $X$, the points of punctured tangent spaces $T_{x} - \{0\}$, $x \in \{0,\xi^{1},\ldots,\xi^{N},\infty\}$ called tangential base-points (\cite{Deligne}, \S15), and the canonical base-point $\omega_{\DR}$ (\cite{Deligne}, (12.4.1)). 
\newline\indent All pro-affine schemes $\pi_{1}^{\un,\DR}(X_{K},y,x)$ are canonically isomorphic as schemes to $\pi_{1}^{\un,\DR}(X_{K},\omega_{\DR})$, these isomorphisms being compatible with the groupoid structure (\cite{Deligne}, \S12). Thus, describing the groupoid  $\pi_{1}^{\un,\DR}(X_{K})$ is reduced to describing $\pi_{1}^{\un,\DR}(X_{K},\omega_{\DR})$, which is done in the next statement.

\begin{Proposition-Definition} Let $\frak{e}$ be the alphabet $\{e_{0},e_{z_{1}},\ldots,e_{z_{r}}\}$, and let $\Wd(\frak{e})$ be the set of words over $\frak{e}$ (including the empty word).
\newline i) The shuffle Hopf algebra over $\frak{e}$, denoted by $\mathcal{O}^{\sh,\frak{e}}$, is the $\mathbb{Q}$-vector space $\mathbb{Q}\langle \frak{e} \rangle = \mathbb{Q}\langle e_{z_{0}},e_{z_{1}},\ldots,e_{z_{r}}\rangle$, which admits $\Wd(\frak{e})$ as a basis, with the following operations :
\newline a) the shuffle product $\sh:\mathcal{O}^{\sh,\frak{e}}\otimes \mathcal{O}^{\sh,\frak{e}} \rightarrow \mathcal{O}^{\sh,\frak{e}}$ defined by, for all words : $(e_{z_{i_{n+n'}}}\ldots e_{z_{i_{n+1}}})\text{ }\sh\text{ }(e_{z_{i_{n}}} \ldots e_{z_{i_{1}}}) =
\sum_{\sigma}
e_{z_{i_{\sigma^{-1}(n+n')}}} \ldots e_{z_{i_{\sigma^{-1}(1)}}}$ where the sum is over permutations $\sigma$ of $\{1,\ldots,n+n'\}$ such that $\sigma(n)<\ldots<\sigma(1)$ and $\sigma(n+n')<\ldots<\sigma(n+1)$.
\newline b) the deconcatenation coproduct $\Delta_{\dec} : \mathcal{O}^{\sh,\frak{e}}\rightarrow \mathcal{O}^{\sh,\frak{e}} \otimes \mathcal{O}^{\sh,\frak{e}}$, defined by, for all words :
$\Delta_{\dec}(e_{z_{i_{n}}}\ldots e_{z_{i_{1}}}) = \sum_{n'=0}^{n} e_{z_{i_{n}}}\ldots e_{z_{i_{n'+1}}} \otimes e_{z_{i_{n'}}} \ldots e_{z_{i_{1}}}$ 
\newline c) the counit $\epsilon : \mathcal{O}^{\sh,\frak{e}} \rightarrow \mathbb{Q}$ sending all non-empty words to $0$.
\newline d) the antipode $S : \mathcal{O}^{\sh,\frak{e}} \rightarrow \mathcal{O}^{\sh,\frak{e}}$, defined by, for all words : 
$S(e_{z_{i_{n}}}\ldots e_{z_{i_{1}}}) = (-1)^{l} e_{z_{i_{1}}}\ldots e_{z_{i_{n}}}$.
\newline ii) (\cite{Deligne}, \S12) The group scheme $\Spec(\mathcal{O}^{\sh,\frak{e}})$ is pro-unipotent and canonically isomorphic to $\pi_{1}^{\un,\DR}(X_{K},\omega_{\DR})$.
\end{Proposition-Definition}

\noindent Since $\Spec(\mathcal{O}^{\sh,\frak{e}})$ is pro-unipotent, its points can be expressed in a canonical way in terms of the dual of the topological Hopf algebra $\mathcal{O}^{\sh,\frak{e}}$. This is written in the next statement, in which, following a common abuse of notation, we denote in the same way the letters $e_{z_{j}}$ and their duals.

\begin{Proposition-Definition} \label{shuffle equation}
i) Let $K\langle\langle  \frak{e}\rangle\rangle = K \langle\langle e_{z_{0}},\ldots,e_{z_{r}} \rangle\rangle$ be the non-commutative $K$-algebra of power series over the variables $e_{z_{0}},\ldots,e_{z_{r}}$ with coefficients in $K$. It is the completion of the universal enveloping algebra of the complete free Lie algebra over the variables $e_{z_{0}},\ldots,e_{z_{r}}$. It thus has a canonical structure of topological Hopf algebra.
\newline We write an element $f \in K\langle \langle \frak{e} \rangle\rangle$ as $f = \sum_{w \in \Wd(\frak{e})} f[w]w$,  where $f[w] \in K$ for all $w$. We have 
\begin{multline} \label{eq:shuffle equation modulo products}
\begin{array}{ll}
\Lie(\mathcal{O}^{\sh,\frak{e}})^{\vee}\otimes_{\mathbb{Q}} K & = \{ f \in K \langle\langle \frak{e} \rangle\rangle \text{ }|\text{ }\forall w\not=\emptyset,w'\not=\emptyset \in \Wd(\frak{e}), f[w\text{ }\sh\text{ }w']=0\} 
\\ & = 
\{\text{ primitive elements of } K \langle\langle \frak{e} \rangle\rangle \}
\end{array}
\end{multline}
The equation above is called the shuffle equation modulo products.
\newline ii) We have a canonical isomorphism of topological Hopf algebras $(\mathcal{O}^{\sh,\frak{e}} \otimes_{\mathbb{Q}} K)^{\vee} = K \langle \langle \frak{e} \rangle\rangle$ and
\begin{multline} \label{eq:shuffle equation} 
\begin{array}{ll} \Spec(\mathcal{O}^{\sh,\frak{e}})(K) & =
\{ f \in K \langle\langle \frak{e} \rangle\rangle \text{ }|\text{ }\forall w,w' \in \Wd(\frak{e}), f[w\text{ }\sh\text{ }w']=f[w]f[w'],\text{ and }f[\emptyset] = 1 \}
\\ & = \{\text{ grouplike elements of } K \langle\langle \frak{e} \rangle\rangle\}
\end{array}
\end{multline}
The equation above is called the shuffle equation.
\end{Proposition-Definition}

\noindent We now review the connection $\nabla_{\KZ}$ associated with $\pi_{1}^{\un,\DR}(X_{K})$ and multiple polylogarithms, viewed first as power series.

\begin{Proposition-Definition} (follows from \cite{Deligne}, \S7.30 and \S12) \label{prop connexion}
\newline i) The connection associated with $\pi_{1}^{\un,\DR}(X)$, called the Knizhnik-Zamolodchikov (for short, KZ) connection of $X$, is the connection on $\pi_{1}^{\un,\DR}(X,\omega_{\DR}) \times X$ defined by
$\nabla_{\KZ} : f \mapsto df - \big(\sum_{i=0}^{r} e_{z_{i}} f \frac{dz}{z-z_{i}} \big)f$.
\newline ii) The coefficients of its horizontal sections are iterated integrals of $\frac{dz}{z-z_{j}}$, $j=0,\ldots,r$(in the sense of Chen \cite{Chen} if $K\hookrightarrow\mathbb{C}$, and in the sense of Coleman \cite{Coleman} if $K\hookrightarrow \mathbb{C}_{p}$ is unramified), and called multiple polylogarithms \cite{Go}.
\newline Assume $K$ is embedded in $\mathbb{C}$ or $\mathbb{C}_{p}$ for $p$ a prime number. For $d \in \mathbb{N}^{\ast}$, $n_{1},\ldots,n_{d} \in \mathbb{N}^{\ast}$, $j_{1},\ldots,j_{d} \in \{1,\ldots,r\}$, let $\Li \big((n_{i});(z_{j_{i}})\big)_{d} \in K[[z]]$ be the formal iterated integral of the sequence of differential forms $(\underbrace{\frac{dz}{z},\ldots,\frac{dz}{z}}_{n_{d}-1},\frac{dz}{z-z_{j_{d}}},\ldots,\underbrace{\frac{dz}{z},\ldots,\frac{dz}{z}}_{n_{1}-1},\frac{dz}{z-z_{j_{1}}})$. Then, for $z\in K$ such that $|z|<1$, we have :
\begin{equation} \Li^{0} \big((n_{i});(z_{j_{i}})\big)_{d}(z) =  \sum_{0<m_{1}<\ldots<m_{d}} \frac{\big( \frac{z_{j_{2}}}{z_{j_{1}}}\big)^{n_{1}} \ldots \big( \frac{z}{z_{j_{d}}}\big)^{n_{d}}}{m_{1}^{n_{1}} \ldots m_{d}^{n_{d}}} \in K 
\end{equation}
\end{Proposition-Definition}

\subsubsection{Localization of $(\pi_{1}^{\un,\DR}(\mathbb{P}^{1} - \{0=z_{0},z_{1},\ldots,z_{r},\infty\}),\nabla_{\KZ})$ on a neighborhood of zero}

We now formalize the localization on a neighborhood of $0$ of $\pi_{1}^{\un,\DR}(X_{K}),\nabla_{\KZ}$ (Definition \ref{def of the localization}).
\newline \indent Let $A$ be a ring, and $S$ a multiplicative subset of $A$. The localization of $A$ at $S$ is the ring $A S^{-1}$ representing the subfunctor of $\Hom(A,-)$ defined by the homomorphisms mapping $S$ to units. Explicitly, $A S^{-1}$ is the ring whose elements are sums of elements of the form $x_{1}y_{1}^{-1}x_{2}y_{2}^{-1}\ldots x_{i}y_{i}^{-1}$, with $x_{i} \in A$, $y_{i} \in S$. The representability of the functor above is granted because it is continuous and satisfies the solution set condition. This notion is mostly usual when $A$ is commutative, or if the $(A,S)$ satisfies Ore's conditions, which are a weak variant of the commutativity assumption.
\newline \indent For us, localizing $\pi_{1}^{\un,\DR}(X)$ will mean replacing $\mathcal{O}^{\sh,\frak{e}}$, the Hopf algebra of $\pi_{1}^{\un,\DR}(X_{K},\omega_{\DR})$, regarded as a ring whose multiplication is the concatenation of words, by its localization at the part of non-zero elements (which is multiplicative because it is an integral ring). We define a ring which will have a surjection onto the localization ; this will be practical for writing some results.

\begin{Definition} Let $\frak{e}^{\inv}$ be the alphabet $\{e_{0}^{\inv},e^{\inv}_{z_{1}},\ldots,e^{\inv}_{z_{r}}\}$.
\newline Let $\frak{e} \cup \frak{e}^{\inv}$ be the alphabet $\{e_{0},e_{z_{1}},\ldots,e_{z_{r}},e_{0}^{\inv},e^{\inv}_{z_{1}},\ldots,e^{\inv}_{z_{r}}\}$.
\newline Let $\Wd(\frak{e} \cup \frak{e}^{\inv})$ be the set of words over $\frak{e} \cup \frak{e}^{\inv}$.
\newline Let $K \langle\langle \frak{e} \cup \frak{e}^{\inv} \rangle\rangle$ the non-commutative $K$-algebra of formal power series over the variables equal to the letters of $\frak{e} \cup \frak{e}^{\inv}$.
\end{Definition}

\noindent It is convenient to reformulate the KZ equation $\nabla_{\KZ}(L)=0$ as a fixed-point equation :

\begin{Definition} \label{definition int KZ power series}Let the integration operator $K[[z]][\log(z)]\langle\langle \frak{e} \rangle\rangle \rightarrow K[[z]][\log(z)]\langle\langle \frak{e} \rangle\rangle$ :
$$ \Int_{\KZ} : L \mapsto  \int_{\vec{1}_{0}}^{z} (\frac{dz'}{z'} e_{0} + \frac{dz'}{z'-1}e_{1})L $$
\end{Definition}

\noindent Let $L \in K[[z]][\log(z)]\langle \langle \frak{e} \rangle\rangle$ whose coefficient of $z^{0}$ is $\exp(e_{0} \log(z))$. We have the equivalence 
$\nabla_{\KZ}(L) = 0 \Leftrightarrow \Int_{\KZ}(L) = L$. These conditions are also equivalent to saying that $L$ is the non-commutative generating series of multiple polylogarithms in the sense of Proposition-Definition \ref{prop connexion}. This way to formulate the KZ equation gives rise to the following definition of its localized variant :

\begin{Definition} \label{nabla KZ localized}
Let 
$$ \Int^{\loc}_{\KZ} : L \mapsto \bigg( e_{0}^{\inv} z \frac{d}{dz} + e_{1}^{\inv}(z-1) \frac{d}{dz} \bigg) L + \int_{\vec{1}_{0}}^{z} (\frac{dz'}{z'} e_{0} + \frac{dz'}{z'-1}e_{1}) L $$
\noindent We say that the equation $\Int_{\loc}^{\KZ}(L)=L$ is the localized version of the equation $\nabla_{\KZ}(L)=0$.
\end{Definition}

\begin{Definition} \label{def of the localization} The localization on a neighborhood of $0$ of $\pi_{1}^{\un,\DR}(X_{K}),\nabla_{\KZ}$ is the data of the $K$-algebra $K\langle\langle \frak{e} \cup \frak{e}^{\inv} \rangle\rangle$, the inclusion $\pi_{1}^{\un,\DR}(X,\omega_{\DR})(K) \subset K\langle\langle \frak{e} \cup \frak{e}^{\inv} \rangle\rangle$, and the operator $\Int^{\loc}_{\KZ}$.
\end{Definition}

\begin{Proposition-Definition} \label{localized Li} i) The localized KZ equation has a unique solution $\Li^{\loc}_{0}$ such that $L \in K[[z]][\log(z)]\langle \langle \frak{e} \cup \frak{e}^{\inv} \rangle\rangle$ whose coefficient of $z^{0}$ is $\exp(e_{0} \log(z))$.
\newline We call localized $p$-adic multiple polylogarithms the coefficients $\Li_{0}^{\loc}[w] \in K[[z]][\log(z)]$.
\newline ii) $\Li_{0}^{\loc}$ viewed as an element of $K \langle\langle \frak{e} \cup \frak{e}^{\inv} \rangle\rangle$ descends to an element of the $(K\langle\langle \frak{e} \rangle\rangle - \{0\})^{-1}K\langle\langle \frak{e} \cup \frak{e}^{\inv} \rangle\rangle$, and further to $(K\langle\langle \frak{e} \rangle\rangle - \{0\})^{-1}K\langle\langle \frak{e} \rangle\rangle/I_{\comm}$, where $I_{\comm}$ is the ideal generated by the relations $e_{z} e_{z}^{\inv} = e_{z}^{\inv}e_{z}=1$.
\end{Proposition-Definition}

\begin{proof} Clear.
\end{proof}

\noindent The next statement is a generalization of the expression of the power series expansions of multiple polylogarithms in terms of multiple harmonic sums (\cite{Go}, equation (1)) :

\begin{Proposition-Definition} \label{loc mhs} i) We call localized multiple harmonic sums the following numbers, for $d \in \mathbb{N}^{\ast}$ $n_{1},\ldots,n_{d} \in \mathbb{Z}$, $j_{1},\ldots,j_{d+1} \in \mathbb{Z}/N\mathbb{Z}$, $m \in \mathbb{N}^{\ast}$ :
$$ \frak{h}_{m} \big((n_{i});(\tilde{\xi}_{N}^{j_{i}})\big)_{d} =  \sum_{0<m_{1}<\ldots<m_{d}<m} \frac{\big( \frac{z_{j_{2}}}{z_{j_{1}}}\big)^{m_{1}} \ldots \big( \frac{z_{j_{d+1}}}{z_{j_{d}}}\big)^{m_{d}}
\big(\frac{1}{z_{j_{d+1}}}\big)^{m} }{m_{1}^{n_{1}} \ldots m_{d}^{n_{d}}} $$ 
\noindent and weighted localized multiple harmonic sums the numbers 
$$ \har_{m} \big((n_{i});z_{j_{i}})\big)_{d} = m^{n_{1}-\tilde{n}_{1}+\ldots+n_{d}-\tilde{n}_{d}} \frak{h}_{m} \big((n_{i});(\xi_{N}^{j_{i}})\big)_{d} $$
\noindent ii) Assume that $K$ is embedded in $\mathbb{C}$ or in $\mathbb{C}_{p}$.
For all $n_{d},\tilde{n}_{d},\ldots,n_{1},\tilde{n}_{1} \in \mathbb{N}^{\ast}$, $j_{1},\ldots,j_{d} \in \mathbb{Z}/N\mathbb{Z}$, for all $z \in K$, $|z|<1$, the series below is absolutely convergent and we have :
$$ \Li_{0}^{\loc}\big[ e_{0}^{n_{d}-1}(e_{0}^{\inv})^{\tilde{n}_{d}}e_{z_{j_{d}}} \ldots e_{0}^{n_{1}-1}(e_{0}^{\inv})^{\tilde{n}_{1}}e_{z_{j_{1}}} \big] 
= \sum_{0<m_{1}<\ldots<m_{d}} \frac{\big( \frac{z_{j_{2}}}{z_{j_{1}}}\big)^{m_{1}} \ldots \big( \frac{z}{z_{j_{d}}}\big)^{m_{d}}}{m_{1}^{n_{1}-\tilde{n}_{1}} \ldots m_{d}^{n_{d}-\tilde{n}_{d}}} $$ 
\noindent The coefficients of $\Li^{\loc}_{0}[w]$ with $w$ of the form $\tilde{w}e_{z}^{\inv}$, $z \in \{z_{0},z_{1},\ldots,z_{r}\}$, are equal to $0$.
\end{Proposition-Definition}

\subsection{Computation of the localization}

We define two "localization maps", expressing the localization of $(\pi_{1}^{\un,\DR}(\mathbb{P}^{1} - \{0=z_{0},z_{1},\ldots,z_{r},\infty\}),\nabla_{\KZ})$ on a neighborhood of zero, in terms of iterated integrals and algebraic functions. For certain statements, we restrict for simplicity the study to the localization at the multiplicative part generated by $e_{0}$.

\subsubsection{The localization map for multiple polylogarithms $\loc^{\smallint}$}

We write the coefficients of $\Li^{\KZ}_{\loc}$ as $\mathbb{Q}$-linear combinations of products of algebraic functions on $\mathbb{P}^{1} - \{0,\mu_{N},\infty\}$ by coefficients $\Li^{\KZ}_{\loc}[w]$ with $w \in \mathcal{W}(e_{X_{K}})$.
In the next statement, we view $\Li_{\loc}$ as a map $\mathcal{O}^{\sh,e}_{\loc} \rightarrow \mathbb{C}[[z]]$, and we view $\Li$ as a map 
$\mathcal{O}^{\sh,e} \rightarrow \mathbb{C}[[z]]$.

\begin{Proposition-Definition} \label{loc for Li} Assume $-1 \in \{z_{1},\ldots,z_{r}\}$ (otherwise replace $X$ by $X'=(\mathbb{P}^{1} - \{0,z_{1},\ldots,z_{r},-1,\infty\})/K$). There exists a map
$$ \loc^{\smallint} : \mathcal{O}^{\sh,e}_{\loc} \rightarrow 
\Gamma(X,\mathcal{O}_{X}) \otimes \mathcal{O}^{\sh,e} $$
such that we have
$$ \Li_{0}^{\loc} = (\Li^{0} \otimes \Gamma(X,\mathcal{O}_{X})) \circ \loc^{\smallint} $$
\end{Proposition-Definition}

\begin{proof} Let us call weight the number of letters of a word $w$ over the alphabet $\frak{e} \cup \frak{e}^{\inv}$. By induction on the weight, we are reduced to prove that for $x \in \{0,z_{1},\ldots,z_{n}\}$, and $w$ a word over $\frak{e}$, and let $n \in \mathbb{N}^{\ast}$. 
Then $\int_{0}^{z} \frac{dz'}{(z'-x)^{n}}\Li[w](z')$ is a $\Gamma(X,\mathcal{O}_{X})$-linear combination of multiple polylogarithms.
\newline If $x \not=0$, we write 
$\frac{1}{(z'-x)^{n}} = \frac{1}{(-x)^{n}}(\frac{1}{1-\frac{z'}{x}})^{n} =\frac{1}{(-x)^{n}}\sum_{l\geq 0} {-n \choose l} \big(\frac{z'}{x}\big)^{l} = \frac{1}{(-x)^{n}}\sum_{l\geq 0} (-1)^{l} {l+n-1 \choose l} \big(\frac{z'}{x}\big)^{l}$. We use that ${l+n-1 \choose l}$ is a polynomial function of $l$.
\newline If $x=0$, integration by parts the KZ equation and induction on the weight reduces us to the case of $w$ is of weight $1$, in which case $\Li[w](z')=\log(z'-x')$ with $x' \in \{0,z_{1},\ldots,z_{r}\}$ ; the result is then proved by another integration by parts.
\end{proof}

\subsubsection{The localization map for multiple harmonic sums $\loc^{\Sigma}$}

\noindent We review from \cite{I-2} the map $\loc^{\Sigma}$ giving an expression of the localized multiple harmonic sums (in the sense of Definition \ref{loc mhs}) in terms of multiple harmonic sums and some polynomials of the upper bound of the domain of iterated summation.

\begin{Definition} \label{definition of connected partition}Let $S$ be a subset of $\mathbb{N}$.
\noindent\newline i) A connected partition of $S$ is a partition of $S$ into segments.
\newline ii) An increasing connected partition of $S$ is a connected partition of $S$ with an order $<$ on the corresponding set of parts of $S$, such that if $C<C'$, we have $j<j'$ in $\mathbb{N}$ for all $j \in C$ and $j' \in C'$.
\newline iii) For a part $P$ of $\{1,\ldots,d\}$ and an increasing connected partition of $P$ and the connected components of $\{1,\ldots,d\} - P$.
\newline iv) Let $\partial S$ be the set of $x \in S$ such that $x+1 \not\in S$ or $x-1 \not\in S$.
\end{Definition}

\noindent We apply the previous definitions to define a way to represent localized words which is adapted to our purposes.

\begin{Definition} \label{la definition du localise}Let $w =\big(t_{d},\ldots,t_{1}) \in \mathbb{Z}^{d}$.
\newline i) Let $\Sign^{-}(w)= \{i \in \{1,\ldots,d\} \text{ | } t_{i}<0\}$, and $\Sign^{+}(w) = \{ i \in \{1,\ldots,d\} \text{ | } t_{i} \geq 0\}$.
\newline ii) Let $r(w) \in \mathbb{N}$ be the number of connected components of $\Sign^{-}(w)$ in the sense of Definition \ref{definition of connected partition}, and we denote these connected components by $[I_{1}(w)+1,J_{1}(w)-1],\ldots,[I_{r(w)}(w)+1,J_{r(w)}(w)-1]$, with $I_{1}(w)<J_{1}(w)<I_{2}(w)<J_{2}(w)<\ldots < I_{r(w)}(w)<J_{r(w)}(w)$. We also write $J_{0} = 0$ and $I_{r(w)+1} = d+1$.
\newline iii) Let us write $t_{i} = n_{i}$ if $t_{i}>0$, and $t_{i} = -l_{i}$ if $t_{i} \geq 0$.
\end{Definition}

For technical reasons which will appear in \S3, we are actually going to replace the localized multiple harmonic sums (Proposition-Definition \ref{loc mhs}) by a variant whose domain of summation involves both strict and large inequalities. Indeed, the following variant of multiple harmonic sums appears in a natural way in the computation on $p$MZV$\mu_{N}$'s.

\begin{Definition} We take the notations of Proposition-Definition \ref{loc mhs} and Definition \ref{la definition du localise} ; let 
\begin{equation}
\label{eq:def tilde multiple harmonic sums}
\widetilde{\frak{h}}_{m}\big( (n_{i});(\xi^{j_{i}})\big)_{d}= \sum_{(m_{1},\ldots,m_{d}) \in \tilde{\Delta}_{w}} \frac{{\big( \frac{\xi^{j_{2}}}{\xi^{j_{1}}}\big)^{m_{1}} \ldots \big( \frac{1}{\xi^{j_{d}}}\big)^{m_{d}}}}{m_{1}^{n_{1}} \ldots m_{d}^{n_{d}}}
\end{equation}
where $w=\big( (n_{i});(\xi^{j_{i}})\big)_{d}$ and
$$ \tilde{\Delta} = \{ (m_{1},\ldots,m_{d}) \in \mathbb{N}^{d}\text{ }|\text{ }0\leq m_{I_{1}(w)} < \ldots < m_{I_{2}(w)-1}\leq m_{I_{2}(w)}<m_{J_{2}(w)} <\ldots \leq m_{I_{r(w)}(w)} < \ldots < m \} $$
\end{Definition}

\begin{Proposition-Definition} \label{numbers mathcal B}There exists a unique sequence $(\mathcal{B}^{\delta,\delta'}_{(m_{i});(\xi^{j_{i}})})$ of elements of $\mathbb{Q}(\xi)$ such that, for all, $m,m'$ we have
\begin{equation} \label{eq:recurrence} \sum_{m<m_{i}<\ldots<m_{j}<m'} \big( \frac{\tilde{\xi}_{N}^{j_{2}}}{\tilde{\xi}_{N}^{j_{1}}}\big)^{m_{1}} \ldots \big( \frac{1}{\tilde{\xi}_{N}^{j_{d}}}\big)^{m_{d}} m_{1}^{l_{1}}\ldots m_{d}^{l_{d}} =
\sum_{\delta,\delta'=0}^{l_{1}+\ldots+l_{d}+d+1}\mathcal{B}^{\delta,\delta'}_{(m_{i});(\xi^{j_{i}})} m^{\delta}m'^{\delta'} 
\end{equation}
\end{Proposition-Definition}

\begin{proof}The existence of these coefficients as well as formulas for them can be obtained by induction on $d$, and by considering the two following equalities, valid for $l \in \mathbb{N}^{\ast}$ : $\sum_{m_{1}=0}^{m-1} m_{1}^{l} =\sum_{\delta=0}^{l+1} \frac{1}{l+1}{l+1 \choose \delta} B_{l+1-\delta}T^{\delta}$, and $\sum_{m_{1}=0}^{m-1} m_{1}^{l}T^{m_{1}}= (T\frac{d}{dT})^{l}(\sum_{m_{1}=0}^{m} T^{m_{1}}) = (T\frac{d}{dT})^{l}( \frac{T^{m}-1}{T-1})$.
\end{proof}

\begin{Proposition-Definition} \label{loc for har n}Let the map 
$$ \loc^{\Sigma} : \mathcal{O}^{\sh,\frak{e}}_{\loc} \longrightarrow  \mathcal{O}^{\sh,\frak{e}} \otimes \mathbb{Q}(\xi)[\frak{m}] $$
\noindent defined recursively as follows : let $[i_{C},j_{C}]$ be a connected component of $\Sign^{-}(w)$. Then applying equation (\ref{eq:recurrence}) gives an expression of the form $\widetilde{\frak{h}}_{m}(w) = \sum_{w'} \frak{h}_{m}(w') P_{w'}(m)$ with, for all $w'$, $depth(w')<depth(w)$.
We define $\loc(w)$ as $\sum_{w'} \loc(w') (1 \otimes P_{w'})$.
Then, we have :
$$ \widetilde{\frak{h}}_{m}(w) = (\frak{h}(m) \times \eval_{m}) (\loc(w)) $$
\noindent where $\frak{h}(m) \times \eval_{m}$ is defined as $\frak{h}(m) \otimes \eval_{m}$ composed with the multiplication of tensor components.
\end{Proposition-Definition}

\noindent In other terms, a localized multiple harmonic sum $\widetilde{\frak{h}}_{m}(w)$ is a $\mathbb{Q}(\xi)$-linear combination of products of (non-localized) multiple harmonic sums by polynomials of $m$.

\begin{Example} Below, $l_{1},l_{2} \in \mathbb{N}$, and $n_{1},n_{2} \in \mathbb{N}^{\ast}$.
\newline i) Depth one and $N=1$ : 
$\tilde{\frak{h}}_{m}(-l_{1})= \sum_{\delta_{1}=1}^{l_{1}+1} \mathcal{B}_{m}^{l_{1}} m^{\delta_{1}} $
\newline ii) Depth two and $N=1$ :  $\tilde{\frak{h}}_{m}(-l_{2},-l_{1}) = \sum_{\delta=1}^{l_{1}+l_{2}+2} \mathcal{B}_{\delta}^{l_{2},l_{1}} m^{\delta} \displaystyle$
\noindent\newline 
$\frak{h}_{m}(n_{2},-l_{1}) = \left\{
\begin{array}{ll} \sum_{\delta_{1}=1}^{l_{1}+1} \mathcal{B}_{\delta_{1}}^{l_{1}} \frak{h}_{m}(n_{2}-\delta_{1}) \text{ if } l_{1}+1 \leq n_{2} 
\\ \sum_{\delta_{1}=1}^{n_{2}-1} \mathcal{B}_{\delta_{1}}^{l_{1}} \frak{h}_{m}(n_{2}-\delta_{1}) 
+ \sum_{\tilde{\delta}_{1}=0}^{l_{1}-n_{2}+1}
\sum_{\delta_{2}=1}^{\delta_{1}-n_{2}+1}
\mathcal{B}_{\delta_{1}}^{l_{1}}\mathcal{B}_{\delta_{2}}^{\delta_{1}-n_{2}} m^{\delta_{2}} \text{ if } l_{1}+1 > n_{2}
\end{array} \right.$
\newline 
$\tilde{\frak{h}}_{m}(-l_{2},n_{1}) =
\left\{\begin{array}{ll} 
\tilde{\frak{h}}_{m}(-l_{2})\mathcal{H}_{m}(n_{1}) - \sum_{\delta_{2}=1}^{l_{2}+1} \mathcal{B}_{\delta}^{l_{2}} \frak{h}_{m}(n_{1}-\delta_{2}) \text{ if } l_{2}+1 < n_{1}
\\ \tilde{\frak{h}}_{m}(-l_{2})\frak{h}_{m}(n_{1})
- \sum_{\delta_{2}=1}^{n_{1}-1} \mathcal{B}_{\delta_{2}}^{l_{1}} \frak{h}_{m}(n_{1}-\delta_{2}) -
\sum_{\tilde{\delta}_{2}=0}^{l_{2}-n_{1}+1}\sum_{\delta_{1}=1}^{\tilde{\delta}_{2}+1} \mathcal{B}_{\tilde{\delta}_{2}+n_{1}}^{(l_{2}-n_{1})+n_{1}} \mathcal{B}_{\delta_{1}}^{\tilde{\delta}_{2}} m^{\delta_{1}} \text{ if }
\\ l_{2}+1 \geq n_{1}
\end{array} \right.$
\end{Example}

The next definitions can be used to give a close formula for the map $\loc^{\Sigma}$.

\begin{Definition} Let $\text{SignPart}^{-}$ be the set of couples $(Sign(w),P^{-})$ where $w$ is as in Definition \ref{la definition du localise} and $P^{-}$ is a connected partition of $\Sign^{-}(w)$. We define a map $T : \text{SignPart}^{-} \rightarrow \{\text{Finite trees}\}$ by sending $(w,P^{-})$ to the tree defined recursively by :
\newline a) the root of the tree is labeled by $(S^{-}(w),S^{+}(w))$
\newline b) consider a vertex of the tree labeled by a couple of parts $(E^{-},E^{+})$ of $\{1,\ldots,d\}$. If $E^{-} \not= \emptyset$ and $E^{+} \not= \emptyset$, then, for each part $P \subset \partial S^{+}(w)$, we draw an arrow starting from $V$ to a new vertex $V'$, and we label $V'$ by the couple $(E^{+} - P,P)$.
\end{Definition}

\begin{Example} Below, we choose for all examples the connected partition of $\Sign^{-}(w)$ made of singletons.
\newline i) In depth one, the two trivial trees $(1)^{-}$ and $(1)^{+}$
\newline ii) In depth two, we have the two trivial trees $(12)^{-}$ and $(12)^{+}$, as well as 
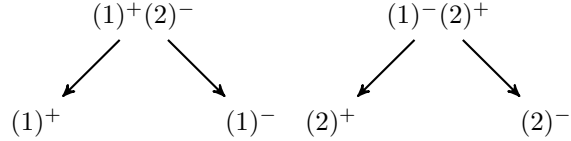
\begin{center}
\begin{tikzpicture}[->,>=stealth',shorten >=1pt,auto,node distance=2cm,
thick,main node/.style={font=\sffamily}]
\node[main node] (1) {$(1)^{+}(2)^{-}$};
\node[main node] (2) [below left of=1] {$(1)^{+}$};
\node[main node] (3) [below right of=1] {$(1)^{-}$};
\path[every node/.style={font=\sffamily\small}]
		(1) edge node [left] {} (2)
		edge [right] node[left] {} (3)
		;
		\end{tikzpicture}
		\begin{tikzpicture}[->,>=stealth',shorten >=1pt,auto,node distance=2cm,
		thick,main node/.style={font=\sffamily}]
		\node[main node] (1) {$(1)^{-}(2)^{+}$};
		\node[main node] (2) [below left of=1] {$(2)^{+}$};
		\node[main node] (3) [below right of=1] {$(2)^{-}$};
		\path[every node/.style={font=\sffamily\small}]
		(1) edge node [left] {} (2)
		edge [right] node[left] {} (3)
		;
		\end{tikzpicture}
	\end{center}
	\noindent iii) In depth three, we have the two trivial trees $(123)^{+}$ and $(123)^{-}$, as well as 
	
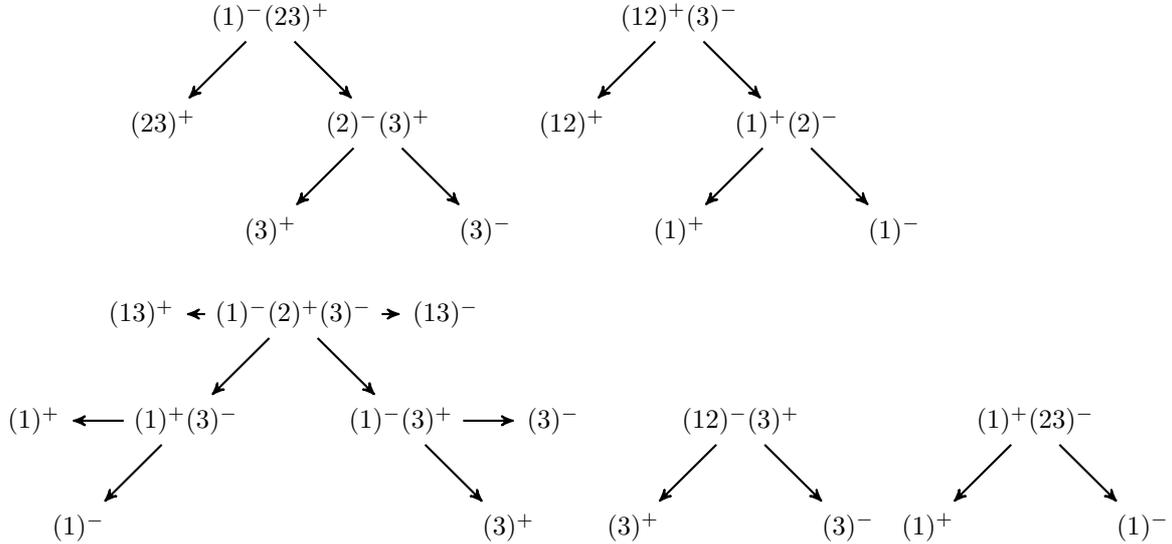
\begin{center}
		\begin{tikzpicture}[->,>=stealth',shorten >=1pt,auto,node distance=2cm,
		thick,main node/.style={font=\sffamily}]
		\node[main node] (1) {$(1)^{-}(23)^{+}$};
		\node[main node] (2) [below left of=1] {$(23)^{+}$};
		\node[main node] (3) [below right of=1] {$(2)^{-}(3)^{+}$};
		\node[main node] (4) [below left of=3] {$(3)^{+}$};
		\node[main node] (5) [below right of=3] {$(3)^{-}$};
		\path[every node/.style={font=\sffamily\small}]
		(1) edge node [left] {} (2)
		edge [right] node[left] {} (3)
		(3) edge node [left] {} (4)
		edge [right] node[left] {} (5) ;
		\end{tikzpicture}
		\begin{tikzpicture}[->,>=stealth',shorten >=1pt,auto,node distance=2cm,
		thick,main node/.style={font=\sffamily}]
		\node[main node] (1) {$(12)^{+}(3)^{-}$};
		\node[main node] (2) [below left of=1] {$(12)^{+}$};
		\node[main node] (3) [below right of=1] {$(1)^{+}(2)^{-}$};
		\node[main node] (4) [below left of=3] {$(1)^{+}$};
		\node[main node] (5) [below right of=3] {$(1)^{-}$};
		\path[every node/.style={font=\sffamily\small}]
		(1) edge node [left] {} (2)
		edge [right] node[left] {} (3)
		(3) edge node [left] {} (4)
		edge [right] node[left] {} (5) ;
		\end{tikzpicture}
		\noindent 
		\newline 
		\newline 
		\begin{tikzpicture}[->,>=stealth',shorten >=1pt,auto,node distance=2cm,
		thick,main node/.style={font=\sffamily}]
		\node[main node] (1) {$(1)^{-}(2)^{+}(3)^{-}$};
		\node[main node] (2) [left of=1] {$(13)^{+}$};
		\node[main node] (3) [right of=1] {$(13)^{-}$};
		\node[main node] (4) [below left of=1] {$(1)^{+}(3)^{-}$};
		\node[main node] (5) [below right of=1] {$(1)^{-}(3)^{+}$};
		\node[main node] (6) [left of=4] {$(1)^{+}$};
		\node[main node] (7) [below left of=4] {$(1)^{-}$};		
		\node[main node] (8) [below right of=5] {$(3)^{+}$};
		\node[main node] (9) [right of=5] {$(3)^{-}$};
		;
		\path[every node/.style={font=\sffamily\small}]
		(1) edge node [left] {} (2)
		edge [right] node[left] {} (3)
		edge node [left] {} (4)
		edge [right] node[left] {} (5)
		(4) edge node [left] {} (6)
		edge [right] node[left] {} (7)
		(5)	edge node [left] {} (8)
		edge [right] node[left] {} (9) ;
		\end{tikzpicture}
		\begin{tikzpicture}[->,>=stealth',shorten >=1pt,auto,node distance=2cm,
		thick,main node/.style={font=\sffamily}]
		\node[main node] (1) {$(12)^{-}(3)^{+}$};
		\node[main node] (2) [below left of=1] {$(3)^{+}$};
		\node[main node] (3) [below right of=1] {$(3)^{-}$};
		\path[every node/.style={font=\sffamily\small}]
		(1) edge node [left] {} (2)
		edge [right] node[left] {} (3)
		;
		\end{tikzpicture}
        \begin{tikzpicture}[->,>=stealth',shorten >=1pt,auto,node distance=2cm,
		thick,main node/.style={font=\sffamily}]
		\node[main node] (1) {$(1)^{+}(23)^{-}$};
		\node[main node] (2) [below left of=1] {$(1)^{+}$};
		\node[main node] (3) [below right of=1] {$(1)^{-}$};
		\path[every node/.style={font=\sffamily\small}]
		(1) edge node [left] {} (2)
		edge [right] node[left] {} (3);
		\end{tikzpicture}
	\end{center}
\end{Example}

\begin{Proposition} (informal version) A close formula for the map $\loc^{\Sigma}$ can be written as a sum over the set of paths from the root to the leaves (sequences of nodes $(N_{1},\ldots,N_{r})$ such that $N_{1}$ is the root, $N_{r}$ is a leaf and, for each $i$, $N_{i+1}$ is a son of $N_{i}$).
\end{Proposition}

\begin{proof} Induction on $d$.
\end{proof}

The explicit version of this proposition formula will appear in the next version of this text, and in the next version of \cite{I-2}.

\subsubsection{Correspondence between the two localizations}

The localization maps $\loc^{\smallint}$ defined in Proposition-Definition \ref{loc for Li} of \S2.2.1, and $\loc^{\Sigma}$ defined in Proposition-Definition \ref{loc for har n} of \S2.2.2 can be related to each other, via the power series expansions of localized multiple polylogarithms in terms of localized multiple harmonic sums (Proposition-Definition \ref{loc mhs}).

\subsection{Multiple polylogarithms, localization and analytic continuation}

We now define analytic continuations of the localized multiple polylogarithms of Proposition-Definition \ref{localized Li}. We have to make a distinction between the complex setting (\S2.3.1) and the $p$-adic setting (\S2.3.2).

\subsubsection{In the complex setting}

We now assume that $K$ is embedded in $\mathbb{C}$. By \cite{Deligne}, there is an isomorphism of comparison between $\pi_{1}^{\un,\DR}(X) \times \mathbb{C}$ and the Betti realization of $\pi_{1}^{\un}(X_{K})\times \mathbb{C}$, and the coefficients of this isomorphism are iterated path integrals in the sense of \cite{Chen}.

\begin{Definition} (Goncharov, \cite{Go}) Let $\gamma$ be a path on $X(\mathbb{C})$ in the generalized sense where the extremities of $\gamma$ are not necessarily points of $X(\mathbb{C})$ but can also be tangential base-points. The multiple polylogarithms are the following functions $(j_{1},\ldots,j_{n} \in \{0,\ldots,r\})$ :
$$ \Li(\gamma)(e_{z_{j_{n}}} \ldots e_{z_{j_{1}}}) = \int_{t_{n}=0}^{1} \gamma^{\ast}(\frac{dt_{n}}{t_{n}-z_{j_{n}}}) \int_{t_{n-1}=0}^{t_{n}} \ldots \gamma^{\ast}(\frac{dt_{2}}{t_{2}-z_{j_{2}}}) \int_{t_{1}=0}^{t_{2}} \gamma^{\ast}(\frac{dt_{1}}{t_{1}-z_{j_{1}}}) $$
\noindent Then 
$$\Li(\gamma) = 1 + \sum_{\substack{n \in \mathbb{N}^{\ast} \\ z_{j_{n}},\ldots,z_{j_{1}} \in \{0,\xi^{1},\ldots,\xi^{N} \}}} \Li(\gamma)(e_{z_{j_{n}}} \ldots e_{z_{j_{1}}}) e_{z_{j_{n}}} \ldots e_{z_{j_{1}}} \in \pi_{1}^{\un,\DR}(X_{K},b,a)(\mathbb{C})$$
\noindent where $a$ and $b$ are the extremities of $\gamma$.
\end{Definition}

\begin{Definition} \label{analytic continuation loc MPL}Let $w$ be a localized word. Let $\gamma$ be a path on $(\mathbb{P}^{1} - \{0,z_{1},\ldots,z_{r},\infty\})(\mathbb{C})$ in the previous sense. We write $\loc^{\smallint}(w) = \sum_{w'} F_{w'} \otimes w'$.
\newline Let $\Li^{\loc}_{\gamma}[w]=\sum_{w'}F_{w'}(z)\Li_{\gamma}(w')$ where $z$ is the endpoint of $\gamma$.
\end{Definition}

\subsubsection{In the $p$-adic setting for $\mathbb{P}^{1} - \{0,\mu_{N},\infty\}$}

The notion of crystalline pro-unipotent fundamental groupoid ($\pi_{1}^{\un,\crys}$) has been defined with three different points of view in \cite{Deligne} \S11, \cite{CLS}, and \cite{Shiho 1}, \cite{Shiho 2}). In our simple example, the three points of view are equivalent and we follow \cite{Deligne} \S11.
\newline We go back to the notations of \S1.1 : $p$ is a prime number, $N \in \mathbb{N}^{\ast}$ is prime to $p$, $\xi_{N}$ is a primitive $N$-th root of unity in $\overline{\mathbb{Q}_{p}}$. We apply \S2.1 and \S2.2 in the case where $K=\mathbb{Q}_{p}(\xi_{N})$, $r=N$, and $(z_{1},\ldots,z_{r})=(\xi_{N}^{1},\ldots,\xi_{N}^{N})$, thus $X=(\mathbb{P}^{1} - \{0,\mu_{N},\infty\})/ K$. According to \cite{Deligne}, $\pi_{1}^{\un,\crys}(\mathbb{P}^{1} - \{0,\mu_{N},\infty\}\text{ }/\text{ }\mathbb{F}_{q})$ is the data of $\pi_{1}^{\un,\DR}(\mathbb{P}^{1} - \{0,\mu_{N},\infty\}\text{ }/\text{ }K)$ plus the Frobenius structure of the KZ connection.  The next definitions refer to Coleman integration as in \cite{Coleman}, \cite{Besser}, \cite{Vologodsky}. They depend on the choice of a determination of the $p$-adic logarithm. The alphabet $\frak{e}$ of the previous paragraphs is now $\{e_{0},e_{\xi^{1}},\ldots,e_{\xi^{N}}\}$ and we denote it by $e_{0 \cup \mu_{N}}$.

\begin{Definition} \label{def Li coleman} (Furusho \cite{Furusho 1} for $N=1$, Yamashita \cite{Yamashita} for any $N$).
\newline We fix a determination $\log_{p}$ of the $p$-adic logarithm. Let $\Li_{p,\KZ}$ be the non-commutative generating series of Coleman functions on $X$ which satisfies $\nabla_{\KZ}\Li_{p,\KZ} = 0$ and 
$\Li_{p,\KZ}(z) \underset{z \rightarrow 0}{\sim} e^{e_{0} \log_{p}(z)}$.
\end{Definition}

The next definition is a generalization of a definition in \cite{FKMT3}.

\begin{Definition} \label{p-adic continuation loc MPL} Let $w$ be a localized word. We write $\loc^{\smallint}(w) = \sum_{w'} F_{w'} \otimes w'$.
\newline Let $\Li^{\loc}_{p,\KZ}[w]=\sum_{w'}F_{w'}\Li_{p,\KZ}[w']$.
\end{Definition}

\section{Localized adjoint $p$-adic multiple zeta values at roots of unity}

We review the definition of $p$-adic multiple zeta values at roots of unity, (\S3.1), of the pro-unipotent $\Sigma$-harmonic action (\S3.2) and we define and study the localized adjoint $p$-adic multiple zeta values at roots of unity (\S3.3, \S3.4).

\subsection{Review on $p$MZV$\mu_{N}$'s and Ad$p$MZV$\mu_{N}$'s}

We review definitions of $p$-adic multiple zeta values at roots of unity.

\begin{Definition} \label{def of tau} (\cite{Deligne Goncharov}, \S5) Let $\tau$ be the action of $\mathbb{G}_{m}(K)$ on $K\langle\langle e_{X_{K}} \rangle\rangle$, that maps $(\lambda,f) \in \mathbb{G}_{m}(K) \times K \langle\langle \frak{e} \rangle\rangle$ to $\sum_{w\in\Wd(\frak{e})} \lambda^{\weight(w)} f[w]w$. 
\end{Definition}

In the next definition, we adopt this convention, which is different from conventions used by some other authors.

\begin{Convention} For $\alpha \in \mathbb{N}^{\ast}$, the Frobenius iterated $\alpha$ times is $\tau(p^{\alpha})\phi^{\alpha}$ where $\phi$ is the Frobenius in the sense of \cite{Deligne}, \S13.6, and, for each $\alpha \in -\mathbb{N}^{\ast}$, the Frobenius iterated $\alpha$ times is $\phi^{-\alpha}$ is in the sense of \cite{Deligne}, \S11. 
\end{Convention}

\begin{Notation} Let $\Pi_{1,0} = \pi_{1}^{\un,\DR}(X_{K},\vec{1}_{1},\vec{1}_{0})$.
\end{Notation}

\noindent A first point of view on the notion of $p$MZV$\mu_{N}$'s uses the canonical De Rham paths evoked in \S2.1.1 :

\begin{Definition} \label{MZV Deligne} (general definition in \cite{I-1}, Definition 2.2.5 ; anterior particular cases : $N=1$, $\alpha=1$, Deligne, Arizona Winter School, 2002 (unpublished) ; $N \in \{1,2\}$, $\alpha=1$, Deligne and Goncharov \cite{Deligne Goncharov} \S5.28 ; $N=1$, $\alpha=-1$ \"{U}nver \cite{Unver MZV} ,\S1 ; any $N$ and $\alpha = \frac{\log(q)}{\log(p)}$, Yamashita \cite{Yamashita}, Definition 3.1 ; any $N$ and $\alpha=-1$ \"{U}nver \cite{Unver cyclotomic}, \S2.2.3).
\newline If $\alpha \in \mathbb{N}^{\ast}$, let
$\Phi_{p,\alpha} = \tau(p^{\alpha})\phi^{\alpha} ( _{\vec{1}_{{\xi^{j}}^{p^{\alpha}}}} 1 _{\vec{1}_{0}}) \in \Pi_{1,0}(K)$ ; if $\alpha \in -\mathbb{N}^{\ast}$, let $\Phi_{p,\alpha} = \phi^{\alpha} ( _{\vec{1}_{{\xi^{j}}^{p^{\alpha}}}} 1 _{\vec{1}_{0}}) \in \Pi_{1,0}(K)$.
\newline For any $\alpha \in \mathbb{N}^{\ast} \cup -\mathbb{N}^{\ast}$, the $p$-adic multiple zeta values at roots of unity are the numbers $\zeta_{p,\alpha}\big((n_{i});(\xi^{j_{i}})\big)= \Phi_{p,\alpha}[e_{0}^{n_{1}-1}e_{\xi^{j_{1}}}\ldots e_{0}^{n_{d}-1}e_{\xi^{j_{d}}}]$, $d \in \mathbb{N}^{\ast}$, and $n_{1},\ldots,n_{d} \in \mathbb{N}^{\ast}$, and $j_{1},\ldots,j_{d} \in \{1,\ldots,N\}$.
\newline For all objects $\ast$ above, and $\alpha = \frac{\log(q)}{\log(p)} \tilde{\alpha}$, let $\ast_{q,\tilde{\alpha}} = \ast_{p,\alpha}$.
\end{Definition}

\indent The second point of view on the notion of $p$MZV$\mu_{N}$'s (which is the general and conceptual one whereas the previous one is more ad hoc) relies on Coleman integration.

\begin{Definition} \label{MZV Coleman} (N=1 : \cite{Furusho 1} Definition 2.17 ; any $N$ Yamashita (\cite{Yamashita} Definition 2.4))
\newline Let $\Phi_{p,\KZ}$ be the unique element of $\Pi_{1,0}(K)$ which is invariant by the Frobenius. 
\newline The numbers $\zeta_{p,\KZ} \big((n_{i});(\xi^{j_{i}})\big) = \Phi_{p,\KZ}[ e_{0}^{n_{d}-1}e_{\xi^{j_{d}}} \ldots e_{0}^{n_{1}-1}e_{\xi^{j_{1}}}] \in K$ are called $p$-adic multiple zeta values at roots of unity.
\end{Definition}

\indent In \cite{I-3}, we have denoted by $\Phi_{p,-\infty} = \Phi_{p,\KZ}$, $\zeta_{p,-\infty} = \zeta_{p,\KZ}$ and we also defined the following variant. Below, the group law $\circ^{\smallint_{0}^{1}}$ on $\Pi_{1,0}$ is the group law denoted by $\circ$ in \cite{Deligne Goncharov}, \S5.12.

\begin{Definition} \cite{I-3}\label{MZV Coleman bis}
Let $\Phi_{p,\infty}$ be the inverse of $\Phi_{p,\KZ}$ for the group law $\circ^{\smallint_{0}^{1}}$. The numbers $\zeta_{p,\infty} \big((n_{i});(\xi^{j_{i}})\big) = \Phi_{p,\infty}[ e_{0}^{n_{d}-1}e_{\xi^{j_{d}}} \ldots e_{0}^{n_{1}-1}e_{\xi^{j_{1}}}] \in K$ are called $p$-adic multiple zeta values at roots of unity.
\end{Definition}

In the Definitions \ref{MZV Deligne}, \ref{MZV Coleman} and \ref{MZV Coleman bis}, we are actually adopting a terminology which differs from the terminologies in other works : the $p$-adic multiple zeta values at roots of unity for $\alpha=-1$ are called cyclotomic $p$-adic multiple zeta values in \cite{Unver cyclotomic}, those for $\alpha = \frac{\log(q)}{\log(p)}$ or $\alpha=-\infty$ are called $p$-adic multiple $L$-values in \cite{Yamashita}.
\newline For any $\alpha,\alpha' \in \mathbb{Z} \cup \{\pm \infty\} - \{0\}$, $\zeta_{p,\alpha}$ and $\zeta_{p,\alpha'}$ can be expressed in terms of each other : for certain particular $\alpha$, this is written in \cite{Furusho 2}, Theorem 2.14, and in \cite{Yamashita} ; and this is expressed in terms of $p$-adic pro-unipotent harmonic actions in \cite{I-3}. We have also defined :

\begin{Definition} \label{def adjoint MZV}(\cite{II-1}) For $\alpha \in \mathbb{N}^{\ast}$, the numbers $\zeta^{\Ad}_{p,\alpha}\big(l;(n_{i});(\xi^{j_{i}})\big) =$ 
\newline $\sum_{j=1}^{N} \xi^{-jp^{\alpha}}(z \mapsto \xi^{j}z)_{\ast} (\Phi_{p,\alpha}^{-1}e_{1}\Phi_{p,\alpha})[e_{0}^{l-1}e_{\xi^{j_{d+1}}}e_{0}^{n_{d}-1}e_{\xi^{j_{d}}}\ldots e_{0}^{n_{1}-1}e_{\xi^{j_{1}}}]$.
\newline We call these numbers the adjoint $p$-adic multiple zeta values at $N$-th roots of unity (Ad$p$MZV$\mu_{N}$'s).
\end{Definition}

In the particular case of $\mathbb{P}^{1} - \{0,1,\infty\}$, these are called adjoint $p$-adic multiple zeta values (Ad$p$MZV's) and are the numbers $\zeta_{p,\alpha}^{\Ad}\big(l;(n_{i})\big)_{d}=(\Phi_{p,\alpha}^{-1}e_{1}\Phi_{p,\alpha})[e_{0}^{l-1}e_{1}e_{0}^{n_{d}-1}e_{1}\ldots e_{0}^{n_{1}-1}e_{1}]$.

\subsection{Review on the $p$-adic pro-unipotent $\Sigma$-harmonic action}

This paragraph is a preliminary for the definition of localized Ad$p$MZV$\mu_{N}$'s in \S3.4, it is a review on definitions in \cite{I-2} \S4, \S5. We adopt the notations of \cite{I-2}.
\newline Below, $\loc^{\vee}$ is the dual of the map $\loc^{\Sigma}$ defined in Proposition-Definition \ref{loc for har n}.
\newline Let $K\langle\langle e_{0}^{\pm 1},e_{\xi^{1}},\ldots,e_{\xi^{N}}\rangle\rangle$ be the set of linear maps $\mathbb{Q}\langle e_{0}^{\pm 1},e_{1}\rangle \rightarrow K$ where $\mathbb{Q}\langle e_{0}^{\pm 1},,e_{\xi^{1}},\ldots,e_{\xi^{N}}\rangle$ is the localization of the non-commutative ring $\mathbb{Q}\langle e_{0},e_{\xi^{1}},\ldots,e_{\xi^{N}}\rangle$ equipped with the concatenation product at the multiplicative part generated by $e_{0}$. The variant $K\langle\langle e_{0}^{\pm 1},e_{\xi^{1}},\ldots,e_{\xi^{N}}\rangle\rangle_{\har}
=\prod_{d \in \mathbb{N},\text{ } (n_{i})_{i} \in \mathbb{Z}^{d},\text{ } (j_{i})_{i} \in (\mathbb{Z}/N\mathbb{Z})^{d+1}} K.((n_{i});(\xi^{j_{i}}))_{d}$ contains, for each $m \in \mathbb{N}$, the generating sequence $\har_{m}^{\loc}$ of localized multiple harmonic sums $\har_{m}(w)$. Below, the subscript $S$ denotes a condition on the $p$-adic valuations of the coefficients defined in \cite{I-2}.
\newline The localized $p$-adic pro-unipotent $\Sigma$-harmonic action defined in \cite{I-2}, \S5 is a map :
$$ (\circ_{\har}^{\Sigma})_{\loc} : (K \langle \langle \frak{e} \rangle\rangle_{\har}^{\Sigma})_{S} \times \Map(\mathbb{N},K \langle\langle e_{0}^{\pm 1},e_{\xi^{1}},\ldots,e_{\xi^{N}} \rangle\rangle_{\har}) \rightarrow 
\Map(\mathbb{N},K\langle\langle e_{0},e_{1}\rangle\rangle_{\har}) $$
\noindent In this paper, we will call it the $p$-adic pro-unipotent $\Sigma$-harmonic action localized at the source. The $p$-adic pro-unipotent $\Sigma$-action defined in \cite{I-2} is the map 

$$ \circ_{\har}^{\Sigma} : K \langle\langle e_{0\cup \mu_{N}} \rangle\rangle_{S} \times \Map(\mathbb{N} \times K \langle\langle e_{0\cup \mu_{N}} \rangle\rangle_{\har}^{\Sigma}) \rightarrow \Map(\mathbb{N} \times K \langle\langle e_{0\cup \mu_{N}} \rangle\rangle_{\har}^{\Sigma}) $$
\noindent defined as $\circ_{\har}^{\Sigma}= (\circ_{\har}^{\Sigma})_{\loc} \circ (\id \times \loc^{\vee})$. We have, in the sense of \cite{I-2}, \S5,
$$ \har_{p^{\alpha}\mathbb{N}} = \har_{p^{\alpha}} (\circ_{\har}^{\Sigma})_{\loc} \har_{\mathbb{N},\loc}^{(p^{\alpha})} = \har_{p^{\alpha}} (\circ_{\har}^{\Sigma})_{\loc} \har_{\mathbb{N},\loc}^{(p^{\alpha})} $$
	
\begin{Example} (\cite{I-2}, \S5) $N=1$, $d=2$ :
$\har_{p^{\alpha}m}(n_{1},n_{2}) = \har_{m}(n_{1},n_{2}) +$
\newline $\sum_{l_{1},l_{2}\geq 0}  
\prod_{i=1}^{2} {-n_{i} \choose l_{i}} m^{n_{i}} \times
\bigg[ \tilde{\frak{h}}_{m}(-l_{1}-l_{2}) \har_{p^{\alpha}}(n_{2}+l_{2},n_{1}+l_{1})
+ \tilde{\frak{h}}_{m}(-l_{2},-l_{1}) \prod_{i=1}^{2} \har_{p^{\alpha}}(n_{i}+l_{i}) \bigg]$
\newline $+ m^{n_{1}+n_{2}} \bigg[ \sum_{l_{1}\geq 0} \har_{p^{\alpha}}(n_{1}+l_{1}) {-n_{1} \choose l_{1}} \widetilde{\frak{h}}_{m}(n_{2},-l_{1})
+ \sum_{l_{2}\geq 0} \har_{p^{\alpha}}(n_{2}+l_{2}) {-n_{2} \choose l_{2}}  \widetilde{\frak{h}}_{m}(-l_{2}, n_{1}) \bigg]$
\newline $= 
\har_{m}(n_{1},n_{2}) +$
\newline $\sum_{t \geq 1} m^{n_{1}+n_{2}+t}
\bigg[ \sum_{\substack{l_{1},l_{2} \geq 0 \\ l_{1}+l_{2} \geq t-2}} 
\mathcal{B}_{t}^{l_{2},l_{1}} 
\prod_{i=1}^{2} {-n_{i} \choose l_{i}} \har_{p^{\alpha}}(n_{i}+l_{i})+
\sum_{\substack{l_{1},l_{2} \geq 0 \\ l_{1}+l_{2} \geq t-1}} 
\mathcal{B}_{t}^{l_{1}+l_{2}} 
\prod_{i=1}^{2} {-n_{i} \choose l_{i}} \har_{p^{\alpha}}(n_{2}+l_{2},n_{1}+l_{1})
\bigg] + \sum_{t \geq 1} m^{n_{2}+n_{1}+t} 
\sum_{l \geq t-1} \bigg[ {-n_{1} \choose l+n_{2}} \mathcal{B}_{t}^{l+s_{2},-n_{2}} - {-n_{2} \choose l+s_{1}} \mathcal{B}_{t}^{l+n_{1},-n_{1}} \bigg]
$
\newline $- m^{n_{2}+n_{1}} \bigg[ \sum_{l_{1} \geq n_{2}-1} \mathcal{B}_{n_{2}}^{l_{1}} {-n_{1} \choose l_{1}} \har_{p^{\alpha}}(n_{1}+l_{1}) 
- \sum_{l_{2} \geq n_{1}-1} \mathcal{B}_{n_{1}}^{l_{2}} {-n_{2} \choose l_{2}} \har_{p^{\alpha}}(n_{2}+l_{2})  \bigg]$
\newline $+ \sum_{\substack{ 1 \leq t < n_{2} \\ l \geq t-1}}
m^{n_{1}+t} \har_{m}(n_{2}-t) \mathcal{B}_{t}^{l} {-n_{1} \choose l} \har_{p^{\alpha}}(n_{1}+l) - \sum_{\substack{1 \leq t < n_{1} \\ l' \geq t-1}}
m^{n_{2}+t} \har_{m}(n_{1}-t) \mathcal{B}_{t}^{l_{2}} {-n_{2} \choose l'} \har_{p^{\alpha}}(n_{2}+l')$
\end{Example}

\subsection{Localized $p$MZV$\mu_{N}$'s : the point of view of Frobenius-invariant paths}

The problem which we want to tackle is to define $p$MZV$\mu_{N}$'s at indices $\big((n_{i});(\xi^{j_{i}})\big)_{d}$ such that $n_{1},\ldots,n_{d}$ are not necessarily $>0$.
\newline In this paragraph, we consider the notion of $p$MZV$\mu_{N}$'s in the sense of Coleman integration (Definition \ref{MZV Coleman}). 
\newline A partial solution to our problem is already given by Furusho, Komori, Matsumoto and Tsumura, using Vologodsky's version of Coleman integration \cite{Vologodsky}. We reformulate it with our terminologies.

\begin{Proposition-Definition} (Furusho, Komori, Matsumoto, Tsumura, \cite{FKMT3})
Assume that $\xi_{N}^{j_{1}}\not=1, ,\ldots,\xi_{N}^{j_{d}}\not=1$. Then $\Li_{p,\KZ}^{\loc}\big((n_{i});(\xi^{j_{i}})\big)_{d}(z)$ has an asymptotic expansion in $K[\log_{p}(1-z)]$ when $z \rightarrow 1$.
\newline The constant coefficient of this power series expansion is a generalized $p$-adic multiple zeta value at $N$-th roots of unity.
\end{Proposition-Definition}

In \cite{FKMT3}, only the inversion of the integration operator associated with $e_{0}$ is considered, whereas here we invert the integration operators associated with all letters $e_{z_{i}}$ (\S2). Thus, the definition above can be extended to our more general framework.
\newline\indent However, even with this generalization, the answer is only partial. For indices which do not necessarily satisfy the hypothesis $\xi_{N}^{j_{1}}\not=1, ,\ldots,\xi_{N}^{j_{d}}\not=1$, we have a different asymptotic expansion :

\begin{Lemma} Each function  $\Li_{p,\KZ}^{\loc}[w]$ admits, when $z \rightarrow 1$ and $z \in K$, an asymptotic expansion in the ring $K[\frac{1}{(z-1)}][\log_{p}(1-z)]$.
\end{Lemma}

\begin{proof} This follows from Definition 2.23 and auxiliary results to Furusho's definition of $p$MZVs (\cite{Furusho 1}, Theorem 2.13 to Theorem 2.18, and Theorem 3.15).
\end{proof}

It may be tempting to define a notion of regularized $p$MZV$\mu_{N}$'s by regularizing brutally the asymptotic expansion above and taking the constant term with respect to both $\log(1-z)$ and $\frac{1}{1-z}$. However, this definition is not relevant. It would imply that $\zeta_{p,\KZ}(-n)$ is zero for all $n \in \mathbb{N}^{\ast}$, whereas we expect non-zero values in odd weights, corresponding to the values at negative integers of the Riemann zeta function. These observations motivate to consider the point of view on $p$MZV$\mu_{N}$'s in terms of canonical De Rham paths (Definition \ref{MZV Deligne}), which we will do in the next paragraph.

\subsection{Localized Ad$p$MZV$\mu_{N}$'s : the point of view of canonical De Rham paths and pro-unipotent harmonic actions}

As in the previous papers, we are going to replace the Frobenius by the harmonic Frobenius in the sense of \cite{I-2}, to use the $p$-adic pro-unipotent harmonic actions, and to replace $p$MZV$\mu_{N}$'s by adjoint $p$MZV$\mu_{N}$'s.
\newline\indent What we want to define is numbers $\zeta^{\Ad}_{p,\alpha}\big( \begin{array}{c} \xi^{j_{1}}, \ldots, \xi^{j_{d}} \\ l;n_{1},\ldots,n_{d} \end{array} \big)$ with $n_{1},\ldots,n_{d}$ of any sign.
\newline We are going to see that this approach will give us a solution to the problem observed in \S3.3, defining implicitly.
\newline The map $(\circ_{\har}^{\Sigma})_{\loc}$ mentioned in \S3.2 is defined by lifting an equation involving multiple harmonic sums. We now define a $p$-adic pro-unipotent $\Sigma$-harmonic action involving a localization both at the source and at the target, by a similar procedure.

\begin{Proposition-Definition} \label{loc action} Let the map 
$$ \circ_{\har}^{\Sigma,\loc,\loc} : K \langle \langle e_{0 \cup \mu_{N}} \rangle\rangle \times \Map(\mathbb{N},K\langle\langle e_{0 \cup \mu_{N}} \rangle \rangle_{\har}^{\Sigma}) \rightarrow \Map(\mathbb{N},K\langle\langle e_{0 \cup \mu_{N}} \rangle \rangle_{\har}^{\Sigma}) $$ 
\noindent called the $p$-adic pro-unipotent $\Sigma$-harmonic action localized at the source and at the target be the map defined by extending the following procedure, used for defining $(\circ_{\har}^{\Sigma})_{\loc}$, to localized multiple harmonic sums : we consider a multiple harmonic sum, whose domain of summation is defined by inequalities of the form $0<m_{1}<\ldots<m_{d}<p^{\alpha}m$. We write the Euclidean division of each $m_{i}$ by $p^{\alpha}$ : $m_{i}=p^{\alpha}u_{i}+r_{i}$ and we express the domain of summation in terms of the $u_{i}$'s and $r_{i}$'s. Then, we write $m_{i}^{-n_{i}} = r_{i}^{-n_{i}}\sum_{l_{i}\geq 0} {-n_{i} \choose l_{i}} \big( \frac{p^{\alpha}u_{i}}{r_{i}} \big)^{l_{i}}$. The map $\circ_{\har}^{\Sigma,\loc,\loc}$ is the natural essentialization of the equation relating localized multiple harmonic sums which appears.
\newline Let the $p$-adic pro-unipotent $\Sigma$-harmonic action localized at the target be the map $\circ_{\har}^{\Sigma,\ast,\loc} = \circ_{\har}^{\Sigma,\loc,\loc} \circ (id \times \loc^{\vee})$. We have :
\begin{equation} \label{eq: the new equation} \har_{p^{\alpha}\mathbb{N},\loc}  = \har_{p^{\alpha}} \circ_{\har,\ast,\loc}^{\Sigma} \har_{\mathbb{N}}^{(p^{\alpha})} 
\end{equation}
\end{Proposition-Definition}

\begin{proof} Similar to the previous statements from \cite{I-2}, \S5 reviewed in \S3.2.
\end{proof}

\noindent We can now recuperate the localized Ad$p$MZV$\mu$'s, as the coefficients of the term of depth $0$ in the $p$-adic pro-unipotent $\Sigma$-harmonic action localized at the target.

\begin{Definition} \label{def localized pMZV} Let $\alpha \in \mathbb{N}^{\ast}$. 
For any localized word $(l;(n_{i});(\xi^{j_{i}}))_{d}$, let us consider expression of $\har_{p^{\alpha}m}\big((n_{i});(\xi^{j_{i}})\big)_{d}$ in terms of $m$, $\har_{m}$ and $\har_{p^{\alpha}}$ given by equation (\ref{eq: the new equation}).	
\newline Let $\zeta_{p,\alpha}^{\Ad}(l;(n_{i});(\xi^{j_{i}}))_{d}$ be the coefficient of $\xi^{jm}m^{l}\har_{m}(\emptyset)$ (where $j$ is the unique element of $\mathbb{Z}/N\mathbb{Z}$ such that such a term appears in the expression).
\newline The weight of an index 
$\big( \begin{array}{c} \xi^{j_{1}}, \ldots, \xi^{j_{d}} \\ l;n_{1},\ldots,n_{d} \end{array} \big)$ is $l+n_{d}+\ldots+n_{1}$.
\end{Definition}

\noindent We now focus on a particular case :

\begin{Definition} \label{def totally negative} The totally negative Ad$p$MZV$\mu_{N}$'s are the numbers $\zeta_{p,\alpha}^{\Ad}(l;(n_{i});(\xi^{j_{i}})$ with $n_{i} \leq 0$ for all $i$.
\end{Definition}

\noindent The next proposition is an analogue of the fact that the desingularized values of multiple zeta functions at tuples of negative integers are rational numbers, having a natural expression as polynomials of Bernoulli numbers.

\begin{Proposition}
The totally negative Ad$p$MZV$\mu_{N}$'s are elements of $\mathbb{Q}(\xi)$.
\newline They can be non-zero only if $1 \leq l+n_{1}+\ldots+n_{d} \leq n_{1}+\ldots+n_{d} +d$, i.e. $1 - (n_{1}+\ldots+n_{d}) \leq l \leq d$.
\end{Proposition}

\begin{proof} Follows directly from Proposition-Definition \ref{numbers mathcal B}
and Definition \ref{def localized pMZV}.
\end{proof}

In the next statement, we write formulas for some examples of the totally negative localized Ad $p$MZV$\mu_{N}$'s for $\mathbb{P}^{1} - \{0,1,\infty\}$.

\begin{Example} Depth one and two, $N=1$. Let $n_{1},n_{2} \in \mathbb{N}$.
\begin{equation}
\zeta_{p,\alpha}^{\Ad}(l+n_{1};-n_{1}) 
= \left\{
\begin{array}{ll}  \displaystyle (p^{\alpha})^{-n_{1}}\mathcal{B}_{l}^{n_{1}} = \frac{{n_{1}+1 \choose l_{1}}}{n_{1}+1} B_{n_{1}+1-l_{1}} \text{ if } 1 \leq l \leq n_{1}+1
\\ 0 \text{ otherwise}
\end{array} \right. 
\end{equation}
\begin{multline}
\zeta_{p,\alpha}^{\Ad}(l+n_{2}+n_{1};-n_{2},-n_{1})
\\ = \left\{
\begin{array}{ll} 
(p^{\alpha})^{-n_{2}-n_{1}} \mathcal{B}_{l}^{n_{2},n_{1}} = \sum_{l_{1}=1}^{n_{1}+1} \displaystyle \frac{{n_{1}+1 \choose l_{1}}}{n_{1}+1} \frac{{l_{1}+n_{2}+1 \choose l }}{k_{1}+n_{2}+1} B_{n_{1}+1-l_{1}} B_{l_{1}+n_{1}+1-l} \text{ if } 1 \leq l \leq 2 + n_{1}+n_{2}
\\ 0 \text{ otherwise}
\end{array} \right. 
\end{multline}
\end{Example}

\section{Localized iteration of the harmonic Frobenius}

The map of iteration of the $\Sigma$-harmonic Frobenius introduced in \cite{I-3} gives a canonical expression of multiple harmonic sums of the form $\har_{q^{\tilde{\alpha}}}$ in terms of multiple harmonic sums of the form $\har_{q^{\tilde{\alpha}_{0}}}$, for $(\tilde{\alpha}_{0},\tilde{\alpha}) \in (\mathbb{N}^{\ast})^{2}$, built by using sums of series. We review it (\S4.1) and explain briefly its generalization to the "localized" framework of \S3.

\subsection{Review of the map of iteration of the harmonic Frobenius}

Let $\Lambda$ and $\textbf{a}$ be formal variables. For $n \in \mathbb{N}^{\ast}$, let $\pr_{n} : K \langle \langle \frak{e} \rangle\rangle \rightarrow K \langle\langle \frak{e} \rangle\rangle$ be the map of "projection onto the terms of weight $n$" i.e. the sequence $(\pr_{n})_{n \in \mathbb{N}}$ is characterized by : for all $f \in K \langle \langle \frak{e}\rangle\rangle$, and $\lambda \in K^{\ast}$, $\tau(\lambda)(f) = \sum_{n \in \mathbb{N}} \pr_{n}(f)\lambda^{n}$.
\newline Let $(\tilde{\alpha}_{0},\tilde{\alpha}) \in (\mathbb{N}^{\ast})^{2}$ such that $\tilde{\alpha}_{0} | \tilde{\alpha}$.
\newline Below, we are using notations of \cite{I-3}. There exists an explicit map, the $\Sigma$-harmonic iteration of the Frobenius $$(\widetilde{\text{iter}}_{\har}^{\Sigma})^{\textbf{a},\Lambda} : (K\langle\langle \frak{e} \rangle \rangle_{\har}^{\Sigma})_{S} \rightarrow K[[\Lambda^{\textbf{a}}]][\textbf{a}](\Lambda)\langle\langle \frak{e} \rangle\rangle^{\smallint_{0}^{1}}_{\har} $$
\noindent such that, the map $(\text{iter}_{\har}^{\Sigma})^{\frac{\alpha}{\alpha_{0}},p^{\alpha_{0}}} : (K\langle\langle \frak{e}\rangle \rangle_{\har}^{\Sigma})_{S} \rightarrow K\langle\langle \frak{e}\rangle\rangle^{\smallint_{0}^{1}}_{\har}$ defined as the composition of $\widetilde{\text{iter}}_{\har,\Sigma}^{\textbf{a},\Lambda}$ by the reduction modulo $(\textbf{a}-\frac{\alpha}{\alpha_{0}},\Lambda-p^{\alpha_{0}})$, satisfies,
\begin{equation} \label{eq:third of I-3}
\har_{q^{\tilde{\alpha}}} = \iter_{\har,\Sigma}^{\frac{\alpha}{\alpha_{0}},p^{\alpha_{0}}} (\har_{q^{\tilde{\alpha}_{0}}})
\end{equation}
\noindent This was used in \cite{I-3} to study the iterated Frobenius as a function of its number of iterations, with the application to have a natural indirect explicit computation of the $p$MZV$\mu_{N}$'s associated with Frobenius-invariant paths (Definition \ref{MZV Coleman}, Definition \ref{MZV Coleman}).

\subsection{Generalization to the localized setting}

\begin{Definition} \label{iter localized} Let the localized iteration of the $\Sigma$-harmonic Frobenius be the map $(\text{iter}_{\har}^{\Sigma})^{\frac{\alpha}{\alpha_{0}},p^{\alpha_{0}}}$ composed with the map $\loc^{\vee}$, dual of the map $\loc^{\Sigma}$ of Proposition-Definition 2.15.
\end{Definition}

\noindent Alternatively, we can construct this map by generalizing the procedure of \cite{I-3} for defining the localized iteration of the $\Sigma$-harmonic Frobenius. Taking a multiple harmonic sums whose domain of summation is of the form $0<m_{1}<\ldots<m_{d}<q^{\tilde{\alpha}}$, we introduce the new parameters $v_{1},\ldots,v_{d}$ equal respectively to the $q$-adic valuations of $m_{1},\ldots,m_{d}$ ; we write $m_{i}= q^{v_{i}}(qu_{i}+r_{i})$ with $r_{i} \in \{1,\ldots,q^{\tilde{\alpha}}-1\}$. We rewrite the domain of summation defined by inequalities $m_{1}<\ldots<m_{d}$ in terms of $v_{i}$'s, $u_{i}$'s and $r_{i}$'s, and sum over these new variables.
\newline 
\newline In the next version of this paper, we will use this map to define a generalization to the localized setting of the Ad$p$MZV$\mu_{N}$'s in the sense of Definition \ref{MZV Coleman} and Definition \ref{MZV Coleman bis}, i.e. the $p$MZV$\mu_{N}$'s in the sense of Coleman integration. 

\section{Localization and algebraic relations}

\subsection{Generalities}

In this paragraph we take the context of \S2.1 : $\pi_{1}^{\un,\DR}(X)$ where $X=(\mathbb{P}^{1} - \{0=z_{0},z_{1},\ldots,z_{r},\infty\})/K$, $K$ being a field of characteristic zero.

\subsubsection{Localization and shuffle equation}

The shuffle equation (\ref{shuffle equation}) is satisfied by multiple polylogarithms (Proposition-Definition \ref{prop connexion}) : namely, we have for all words $w,w'$ on the alphabet $\frak{e}$, $\Li[w\text{ }\sh\text{ }w'] = \Li[w]\Li[w']$.
If we apply $(z-z_{i})\frac{d}{dz}$ a certain number of times to this equation (where $i \in \{0,\ldots,r\}$), since this operator is a derivation, we obtain a variant of the shuffle relation which applies to certain localized multiple polylogarithms. The right-hand side of the relation obtained in this way is encoded by the following generalization of the deconcatenation coproduct $\Delta_{\dec}$ below, which is well-defined on a quotient of the space of words.

\begin{Definition} i) Let $i \in \{0,\ldots,r\}$. Let $e_{z_{i}}^{-\mathbb{N}}\mathbb{Q}\langle \frak{e}\rangle$ be the $\mathbb{Q}$-vector space freely generated by words of the form $e_{z_{i}}^{l}w$ with $w$ a word over $\frak{e}$ and $l \in \mathbb{Z}$. Let $I_{\loc}$ be the ideal of $e_{z_{i}}^{-\mathbb{N}}\mathbb{Q}\langle \frak{e}\rangle$ generated by the relations $\sum_{(w'_{1},w'_{2})\text{ }|\text{ }w'_{1}w'_{2}=w'} w'_{1} \otimes w'_{2} = \sum_{(w_{1},w_{2})\text{ }|\text{ }w_{1}w_{2}=e_{0}^{l}w'} \sum_{m=0}^{l} {l \choose m} e_{0}^{-m}w_{1} \otimes e_{0}^{-(l-m)}w_{2}$ for $w'$ word over $e_{0 \cup \mu_{N}}$ and $l \in \mathbb{N}$.
\newline ii) Let 
$\Delta_{\dec} : \mathbb{Q}\langle e_{0},e_{0}^{-1},e_{\mu_{N}}\rangle / I_{\loc} \rightarrow \mathbb{Q}\langle e_{0},e_{0}^{-1},e_{\mu_{N}}\rangle / I_{\loc} \otimes \mathbb{Q}\langle e_{0},e_{0}^{-1},e_{\mu_{N}}\rangle / I_{\loc}$ be the linear map defined by, for all words $w$ over $e_{0 \cup \mu_{N}}$ :
$\Delta_{\dec}(e_{z_{i}}^{-l}w) = 
\sum_{(w_{1},w_{2})\text{ }|\text{ }w_{1}w_{2}=w} \sum_{m=0}^{l} {l \choose m} e_{z_{i}}^{-m}w_{1} \otimes e_{0}^{-(l-m)}w_{2} $
\end{Definition}

\begin{Definition} Let $K\langle\langle e_{z_{i}}^{-\mathbb{N}} \frak{e} \rangle\rangle$ be the set of linear maps $e_{z_{i}}^{-\mathbb{N}}\mathbb{Q}\langle \frak{e} \rangle \rightarrow K$.
\end{Definition}

\subsubsection{Localization quasi-shuffle relation}

The quasi-shuffle relation \cite{Hoffman} is a consequence of the fact that a product of two sets
$\{(m_{1},\ldots,m_{d}) \in \mathbb{N}^{d}\text{ }|\text{ }0<m_{1}<\ldots<m_{d}<m\}$ and
$\{(m'_{1},\ldots,m_{d'}) \in \mathbb{N}^{d'}\text{ }|\text{ }0<m'_{1}<\ldots<m'_{d}<m\}$ can be written canonically as a disjoint union of sets of the same type in $\mathbb{N}^{r}$, $r \in \{\max(d,d'),\ldots,d+d'\}$. For example, if $d=d'=1$ :
\begin{multline} \{ m_{1} \in \mathbb{N}\text{ }|\text{ }0<m_{1}\}
\times \{ m'_{1} \in \mathbb{N}\text{ }|\text{ }0<m'_{1}\}
= \bigg( \{(m_{1},m'_{1}) \in \mathbb{N}^{2}|\text{ }0<m_{1}<m'_{1}<m \} \\ \amalg 
\{(m_{1},m'_{1}) \in \mathbb{N}^{2}\text{ }|\text{ }0<m'_{1}<m_{1}<m \}
\amalg
\{(m_{1},m'_{1}) \in \mathbb{N}^{2}\text{ }|\text{ }0<m_{1}=m'_{1}<m\} \bigg)
\end{multline}
\noindent The quasi-shuffle relation of MZV$\mu_{N}$'s is obtained by applying this equality to multiple harmonic sums and taking the limit $m \rightarrow \infty$ in $\mathbb{C}$. Example : $\zeta(n)\zeta(n') = \zeta(n,n') + \zeta(n',n) + \zeta(n+n')$.
\newline It can be encoded in the form $\frak{h}_{m}(w) \frak{h}_{m}(w') = \frak{h}_{m}(w \ast w')$ where $\ast$ is a bilinear map $\mathcal{O}^{\sh,e_{0 \cup \mu_{N}}} \times \mathcal{O}^{\sh,e_{0 \cup \mu_{N}}} \rightarrow \mathcal{O}^{\sh,e_{0 \cup \mu_{N}}}$ called the quasi-shuffle product \cite{Hoffman}.
\newline The following statement is clear ; it should also be a priori already known and appear in several works, although we do not have references :

\begin{Proposition-Definition} \label{localized adjoint quasi shuffle} There exists an explicit bilinear map $\ast_{\loc} : e_{0}^{-1}\mathcal{O}^{\sh,e_{0 \cup \mu_{N}}} \times e_{0}^{-1}\mathcal{O}^{\sh,e_{0 \cup \mu_{N}}} \rightarrow e_{0}^{-1}\mathcal{O}^{\sh,e_{0 \cup \mu_{N}}}$ (where the factor $e_{0}^{-1}$ means the localization at the multiplicative part generated by $e_{0}$ for the concatenation product), the localized quasi-shuffle product, such that, for localized multiple harmonic sums in the sense of Definition \ref{loc mhs}, we have, for all $w,w'$ localized words as in that Definition, 
$$ \frak{h}_{m}(w)\frak{h}_{m}(w') = \frak{h}_{m}(w \ast_{\loc} w') $$
\end{Proposition-Definition}

\subsection{Application to localized Ad$p$MZV$\mu_{N}$'s}

We consider now the context of \S2.3.2 and \S3 : $\pi_{1}^{\un,\crys}(\mathbb{P}^{1} - \{0,\mu_{N},\infty\})$.
\newline \indent In \cite{II-1} we defined a notion of adjoint double shuffle relations, satisfied by the Ad$p$MZV$\mu_{N}$'s of \ref{def adjoint MZV}. This includes a notion of adjoint quasi-shuffle relations. In \cite{II-2}, we have showed that the adjoint quasi-shuffle relations of Ad$p$MZV$\mu_{N}$'s can be retrieved by the formulas of part I involving the pro-unipotent harmonic actions. Here is a variant for the localized Ad$p$MZV$\mu_{N}$'s introduced in 
\ref{def localized pMZV}.

\begin{Proposition} The localized Ad$p$MZV$\mu_{N}$'s satisfy a canonical family of polynomial equations which generalizes the adjoint quasi-shuffle relations of \cite{II-1}, and which we call the localized adjoint quasi-shuffle relations.
\end{Proposition}

\begin{proof} Same with the proof in \cite{II-2} of the fact that we can retrieve the fact that Ad$p$MZV$\mu_{N}$'s satisfy the adjoint quasi-shuffle relations from the fact that multiple harmonic sums satisfy the quasi-shuffle equation and the equation relating Ad$p$MZV$\mu_{N}$'s and multiple harmonic sums involving the $p$-adic pro-unipotent harmonic action $\circ_{\har}^{\smallint_{0}^{z<<1}}$.
\newline Here, let us write the localized quasi-shuffle relation for the multiple harmonic sums $\har_{p^{\alpha}m}(w)\har_{p^{\alpha}m}(w')=\har_{p^{\alpha}m}(w\ast_{\loc}w')$ ; then, we use that $\har_{p^{\alpha}m}$'s have an expression in terms of $\har_{m}$'s and certain power series of variable $m$ whose coefficients are written in terms of localized Ad$p$MZV$\mu_{N}$'s : equation (\ref{eq: the new equation}). By the linear independence of the $\har_{m}$'s over the ring of overconvergent power series expansion of $m \in \mathbb{N} \subset \mathbb{Z}_{p}$, this implies a family of polynomial equations satisfied by the localized Ad$p$MZV$\mu_{N}$'s which we call as in the statement.
\end{proof}


\begin{thebibliography}{50}
\bibitem[AET]{AET} S. Akiyama, S. Egami, and Y. Tanigawa, \emph{Analytic continuation of multiple zeta-functions and their values at non-positive integers}, Acta Arith. 98 (2001), no. 2, 107-116.
\bibitem[AETbis]{AETbis} A. Alekseev, B. Enriquez, C. Torossian - \emph{Drinfeld associators, braid groups and explicit solutions of the Kashiwara-Vergne equations}, Publ. Math. Inst. Hautes Etudes Sci. No. 112 (2010), 143-189
\bibitem[AT]{AT} S. Akiyama and Y. Tanigawa, \emph{Multiple zeta values at non-positive integers}, Ramanujan J. 5 (2001), no. 4, 327-351.
\bibitem[Bes]{Besser} A.Besser - \emph{Coleman integration using the Tannakian formalism}, Math. Ann. 322 (2002) 1, 19-48.
\bibitem[BF]{Besser Furusho} A.Besser, H.Furusho - \emph{The double shuffle relations for p-adic multiple zeta values}, AMS Contemporary Math, Vol 416, (2006), 9-29.
\bibitem[Ch]{Chen} K. T. Chen - \emph{Iterated path integrals}, Bull. Amer. Math Soc., Vol. 83, number 5 (1977) 831-879
\bibitem[CL]{CLS} B.Chiarellotto, B.Le Stum - \emph{F-isocristaux unipotents} - Compositio Math. 116, 81-110 (1999).
\bibitem[Co]{Coleman} R.Coleman - \emph{Dilogarithms, regulators and $p$-adic $L$-functions} - Inventiones Mathematicae, June 1982, Vol. 69, Issue 2, pp.171-208
\bibitem[D]{Deligne} P.Deligne, \emph{Le groupe fondamental de la droite projective moins trois points}, Galois Groups over $\mathbb{Q}$ (Berkeley, CA, 1987), Math. Sci. Res. Inst. Publ. 16, Springer-Verlag, New York, 1989.
\bibitem[DG]{Deligne Goncharov} P. Deligne, A.B. Goncharov, \emph{Groupes fondamentaux motiviques de Tate mixtes}, Ann. Sci. Ecole Norm. Sup. 38.1 , 2005, pp. 1-56
\bibitem[E]{Essouabri} D. Essouabri, \emph{Singularité des séries de Dirichlet associées à des polynômes de plusieurs variables et applications à la théorie analytique des nombres}, Thèse de doctorat, Univ. Nancy 1, 1995.
\bibitem[F1]{Furusho 1} H. Furusho - \emph{p-adic multiple zeta values I -- p-adic multiple polylogarithms and the p-adic KZ equation}, Inventiones Mathematicae, Volume 155, Number 2, 253-286, (2004).
\bibitem[F2]{Furusho 2} H. Furusho - \emph{p-adic multiple zeta values II -- tannakian interpretations}, Amer.J.Math, Vol 129, No 4, (2007),1105-1144.
\bibitem[FJ]{Furusho Jafari} H.Furusho, A.Jafari - \emph{Regularization and generalized double shuffle relations for p-adic multiple zeta values}, Compositio Math. Vol 143, (2007), 1089-1107.
\bibitem[FKMT1]{FKMT1} H.Furusho, Y.Komori, K.Matsumoto, H.Tsumura - \emph{Desingularisation of complex multiple zeta functions} - American Journal of Mathematics, Volume 139, Number 1, February 2017,
pp. 147-173
\bibitem[FKMT2]{FKMT2} H.Furusho, Y.Komori, K.Matsumoto, H.Tsumura -\emph{Fundamentals of $p$-adic multiple $L$ functions and evaluation of their special values} - Sel. Math. New Ser. 23, (2017), 39-100.
\bibitem[FKMT3]{FKMT3} H.Furusho, Y.Komori, K.Matsumoto, H.Tsumura -\emph{Desingularization of multiple zeta functions of Hurwitz-Lerch type} arXiv:1404.4758, to appear in RIMS Kokyuroku bessatsu
\bibitem[GPZ]{GPZ} L. Guo, S. Paycha, and B. Zhang, \emph{Algebraic Birkhoff factorization and the Euler-Maclaurin formula on
cones}, preprint, arXiv:1306.3420
\bibitem[GZ]{GZ} L. Guo and B. Zhang, \emph{Renormalization of multiple zeta values}, J. Algebra 319 (2008), no. 9, 3770-3809
\bibitem[Go]{Go} A. B. Goncharov - \emph{Multiple polylogarithms and mixed Tate motives} arXiv:0103059
\bibitem[H]{Hoffman} M. Hoffman - \emph{Multiple harmonic series}, Pacific Journal of Mathematics, Vol. 152, No. 2, 1992
\bibitem[J I-1]{I-1} D. Jarossay,  \emph{$p$-adic multiple zeta values at roots of unity and $p$-adic pro-unipotent harmonic actions - I-1 : Direct computation of $p$-adic multiple zeta values at roots of unity}, arXiv:1503.08756
\bibitem[J I-2]{I-2} D. Jarossay, \emph{$p$-adic multiple zeta values at roots of unity and $p$-adic pro-unipotent harmonic actions - I-2 : Indirect computation of $p$-adic multiple zeta values at roots of unity}, arXiv:1501.04893
\bibitem[J I-3]{I-3} D. Jarossay, \emph{$p$-adic multiple zeta values at roots of unity and $p$-adic pro-unipotent harmonic actions - I-3 : The number of iterations of $p$-adic multiple zeta values at roots of unity viewed as a variable}, arXiv:1610.09107
\bibitem[J II-1]{II-1} D. Jarossay, \emph{$p$-adic multiple zeta values at roots of unity and $p$-adic pro-unipotent harmonic actions - II-1 : Adjoint and harmonic variants of algebraic relations of multiple zeta values at roots of unity}, arXiv:1412.5099
\bibitem[J II-2]{II-2} D. Jarossay, \emph{$p$-adic multiple zeta values at roots of unity and $p$-adic pro-unipotent harmonic actions  - II-2 : From algebraic relations of multiple harmonic sums to those of $p$-adic multiple zeta values at roots of unity}, arXiv:1601.01158
\bibitem[J II-3]{II-3} D. Jarossay, \emph{$p$-adic multiple zeta values at roots of unity and $p$-adic pro-unipotent harmonic actions - II-3 : Multiple harmonic values viewed as periods}, arXiv:1601.01159
\bibitem[J III-1]{III-1} D. Jarossay, \emph{$p$-adic multiple zeta values at roots of unity and $p$-adic pro-unipotent harmonic actions  - III-1 : A generalization of $p$-adic multiple zeta values at roots of unity of order divisible by $p$} : arXiv:1708.08009 
\bibitem[J III-2]{III-2} D. Jarossay, \emph{$p$-adic multiple zeta values at roots of unity and $p$-adic pro-unipotent harmonic actions - III-2} in preparation
\bibitem[KZ]{Kaneko Zagier} M.Kaneko, D.Zagier, \emph{Finite multiple zeta values}, in preparation
\bibitem[Ko]{Komori} 
Y. Komori, \emph{An integral representation of multiple Hurwitz-Lerch zeta functions and generalized multiple Bernoulli numbers}, Q. J. Math. 61 (2010), no. 4, 437-496.
\bibitem[M]{Matsumoto} K. Matsumoto, \emph{The analytic continuation and the asymptotic behaviour of certain multiple zeta-functions. I}, J. Number Theory 101 (2003), no. 2, 223-243.
\bibitem[MP]{MP} D. Manchon and S. Paycha, \emph{Nested sums of symbols and renormalized multiple zeta values}, Int. Math. Res. Not. IMRN 2010 (2010), no. 24, 4628-4697.
\bibitem[O]{O} T. Onozuka, Analytic continuation of multiple zeta-functions and the asymptotic behavior at non-positive
 integers, Funct. Approx. Comment. Math. 49 (2013), no. 2, 331-348.
\bibitem[U1]{Unver MZV} S.\"{U}nver - \emph{$p$-adic multi-zeta values}. Journal of Number Theory, 108, 111-156, (2004).
\bibitem[U2]{Unver cyclotomic} S.\"{U}nver - \emph{Cyclotomic p-adic multiple zeta values in depth two}, Manuscripta Mathematica, March 2016, Volume 149, Issue 3-4, pp. 405-441
\bibitem[Sa1]{Sa1} Y. Sasaki, \emph{Multiple zeta values for coordinatewise limits at non-positive integers}, Acta Arith. 136 (2009),
 no. 4, 299-317.
\bibitem[Sa2]{Sa2} Y. Sasaki, \emph{Some formulas of multiple zeta values for coordinate-wise limits at non-positive integers}, New Directions in Value-Distribution Theory of Zeta and L-Functions (R. Steuding and J. Steuding,
eds.), Ber. Math., Shaker Verlag, Aachen, 2009, pp. 317-325.
\bibitem[Sh1]{Shiho 1} A.Shiho - \emph{Crystalline fundamental groups. I. Isocristals on log crystalline site and log convergent site}, J. Math. Soc. Univ Tokyo 7 (2000) no. 4, 509-656
\bibitem[Sh2]{Shiho 2} A.Shiho - \emph{Crystalline fundamental groups. II. Log convergent cohomology and rigid cohomology},  
J. Math. Soc. Univ. Tokyo 9 (2002), no. 1, 1-163
\bibitem[V]{Vologodsky} V.Vologodsky, \emph{Hodge structure on the fundamental group and its application to $p$-adic integration}, Moscow Math. J. 3 (2003), no. 1, 205-247.
\bibitem[W]{Washington} L.C. Washington, \emph{$p$-adic $L$-functions and sums of powers}, Journal of Number Theory 69 (1998), pp. 50-61. 
\bibitem[Y]{Yamashita} G.Yamashita, \emph{Bounds for the dimension of $p$-adic multiple $L$-values spaces}. Documenta Mathematica, Extra Volume Suslin (2010) 687-723
\bibitem[Z]{Zhao} J. Zhao, \emph{Analytic continuation of multiple zeta functions}, Proc. Amer. Math. Soc. 128 (2000), no. 5, 1275-1283.
\end{thebibliography}
\end{document}